\documentclass[10.5pt,a4paper]{article}

\usepackage[leqno]{amsmath}
\usepackage{graphicx}
\usepackage{latexsym}
\usepackage{amsmath,amsfonts,amssymb,amsthm,euscript,makeidx,color,mathrsfs}
\usepackage{enumerate}

\usepackage[colorlinks,linkcolor=blue,anchorcolor=green,citecolor=red]{hyperref}

\def\sqr#1#2{{\vcenter{\vbox{\hrule height.#2pt
              \hbox{\vrule width.#2pt height#1pt \kern#1pt \vrule width.#2pt}
              \hrule height.#2pt}}}}
\def\5n{\negthinspace \negthinspace \negthinspace \negthinspace \negthinspace }
\def\4n{\negthinspace \negthinspace \negthinspace \negthinspace }
\def\3n{\negthinspace \negthinspace \negthinspace }
\def\2n{\negthinspace \negthinspace }
\def\1n{\negthinspace }

\def\dbE{\mathbb{E}}     
\def\dbF{\mathbb{F}}   \def\cF{{\cal F}}  
     
\def\dbH{\mathbb{H}}

\def\dbP{\mathbb{P}}

\def\ss{\smallskip}                
\def\ms{\medskip}                
               
\def\ds{\displaystyle}           
\def\no{\noindent}        \def\q{\quad}                      
\def\ns{\noalign{\ss}}    \def\qq{\qquad}                    
    \def\hb{\hbox}                     
                   
         \def\rf{\eqref}                    
  \def\deq{\triangleq}               
            \def\({\Big (}
\def\les{\leqslant}                  \def\){\Big )}
\def\ges{\geqslant}          \def\[{\Big[}
           \def\]{\Big]}
                   \def\cd{\cdot}

\def\a{\alpha}        \def\G{\Gamma}         
\def\b{\beta}         \def\D{\Delta}

           \def\i{\infty}   
\def\a{\alpha}        \def\G{\Gamma}         
\def\b{\beta}         \def\D{\Delta}

           \def\i{\infty}   
\def\Om{\Omega}         

\def\bde{\begin{definition}\label}  \def\ede{\end{definition}}
\def\be{\begin{equation}}           \def\bel{\begin{equation}\label}   \def\ee{\end{equation}}
\def\bt{\begin{theorem}\label}      \def\et{\end{theorem}}
\def\bc{\begin{corollary}\label}    \def\ec{\end{corollary}}
\def\bl{\begin{lemma}\label}        \def\el{\end{lemma}}
\def\bp{\begin{proposition}\label}  \def\ep{\end{proposition}}
\def\bas{\begin{assumption}\label}  \def\eas{\end{assumption}}
\def\br{\begin{remark}\label}       \def\er{\end{remark}}
\def\bex{\begin{example}\label}     \def\ex{\end{example}}
\def\ba{\begin{array}}              \def\ea{\end{array}}
\def\ben{\begin{enumerate}}         \def\een{\end{enumerate}}

\def\bt{\textcolor{blue}}
\newtheorem{theorem}{Theorem}[section]
\newtheorem{definition}[theorem]{Definition}
\newtheorem{proposition}[theorem]{Proposition}
\newtheorem{corollary}[theorem]{Corollary}
\newtheorem{lemma}[theorem]{Lemma}
\newtheorem{remark}[theorem]{Remark}
\newtheorem{example}[theorem]{Example}

\makeatletter
   
   \@addtoreset{equation}{section}
\makeatother

\begin{document}

\title{\bf Backward doubly stochastic Volterra integral equations and applications to optimal control problems}

\author{Yufeng Shi\thanks{Institute for Financial Studies and School of Mathematics, Shandong University,
Jinan 250100, China (yfshi@sdu.edu.cn). This author was supported by National Science Foundation of China (Grant Nos. 11871309 and 11626247), the Foundation for Innovative Research Groups of National Natural Science Foundation of China (Grant No. 11221061).}\ , \
Jiaqiang Wen\thanks{Department of
Mathematics, Southern University of Science and Technology, Shenzhen, Guangdong, 518055, China
(wenjq@sustech.edu.cn). This author was partially supported by the National Science Foundation of China (Grant No. 11671229).}\ , \ and \
Jie Xiong\thanks{Department of Mathematics, Southern University of
Science and Technology, Shenzhen, Guangdong, 518055, China (xiongj@sustech.edu.cn). This author was supported by Southern University of Science and Technology Start up fund Y01286120 and National Science Foundation of China grants 61873325 and 11831010.}}

\date{}
\maketitle

\vspace{6mm}

\no\bf Abstract: \rm Backward doubly stochastic Volterra integral equations (BDSVIEs, for short) are introduced and studied systematically. Well-posedness of BDSVIEs in the sense of introduced M-solutions is established. A comparison theorem for BDSVIEs is proved. By virtue of the comparison theorem, we derive the existence of solutions for BDSVIEs with continuous coefficients. Furthermore, a duality principle between linear (forward) doubly stochastic Volterra integral equation (FDSVIE, for short) and BDSVIE is obtained. A Pontryagin type maximum principle is also established for an optimal control problem of FDSVIEs.

\ms

\no\bf Key words: \rm Backward doubly stochastic Volterra integral equations; Comparison theorem; Duality principle; Pontryagin maximum principle

\ms

\no\bf AMS subject classifications. \rm 60H10; 60H20; 93E20

\section{Introduction}

Throughout this paper, let $(\Om,\cF,\dbF,\dbP)$ be a complete filtered probability space and $T>0$ be a fixed terminal time. Let $W=\{W(t); 0\les t<\i\}$ and $B=\{B(t); 0\les t<\i\}$  be two mutually independent standard Brownian motion processes, with values respectively in $\mathbb{R}^{d}$ and in $\mathbb{R}^{l}$, defined on $(\Om,\cF,\dbF,\dbP)$.
Denote by $\mathcal{N}$ the class of $\dbP$-null sets of $\mathcal{F}$.
For each $t\in [0,T]$, we define
 \begin{equation}\label{0}
 \mathcal{F}_{t} \deq \mathcal{F}_{t}^{W} \vee \mathcal{F}_{t,T}^{B},
 \end{equation}
 where for any process $\eta(\cd)$,  $\mathcal{F}_{s,t}^{\eta}\deq\sigma\{\eta(r)-\eta(s);s\les r\les t\}\vee \mathcal{N}$ and  $\mathcal{F}_{t}^{\eta}\deq\mathcal{F}_{0,t}^{\eta}$. Denote $\Delta=\{ (t,s)\in[0,T]^{2}| \ t\les s\}$ and $\Delta^{c}=\{ (t,s)\in[0,T]^{2}| \ t>s\}$. Consider the following integral equation on a finite horizon $[0,T]$:
\bel{1}\ba{ll}
\ds Y(t)=\psi(t) + \int_t^T f(t,s,Y(s),Z(t,s),Z(s,t)) ds\\
\ns\ds\qq\q\ + \int_t^T g(t,s,Y(s),Z(t,s),Z(s,t)) d\overleftarrow{B}(s)- \int_t^T Z(t,s) dW(s), \qq t\in [0,T],
\ea\ee
where  $f(\cd)$, $g(\cd)$ and $\psi(\cd)$ are some given maps, and the $dW$-integral is a forward It\^{o}'s integral and the $d\overleftarrow{B}$-integral is a backward It\^{o}'s integral. The above is called a {\it backward doubly stochastic Volterra integral equation} (BDSVIE, for short).
A pair of processes $(Y(\cd),Z(\cd,\cd))$, valued in $\mathbb{R}^{k}\times \mathbb{R}^{k\times d}$ with $Y(s)$ being $\mathcal{F}_{s}$-measurable and $Z(t,s)$ being $\mathcal{F}_{s}$-measurable for $(t,s)\in[0,T]^{2}$,
is called a {\it solution} to the BDSVIE \rf{1} if the equations are satisfied in the usual It\^o's sense.
Maps $f(\cd)$ and $g(\cd)$ are referred to as the {\it generator} of BDSVIE \rf{1}, and process $\psi(\cd)$ is referred to as the {\it free term}.

\ms

On the one hand, when $g(\cd)$ is absent, \rf{1} is reduced to the form
\begin{equation}\label{46}
 Y(t)=\psi(t) + \int_t^T f(t,s,Y(s),Z(t,s),Z(s,t)) ds  - \int_t^T Z(t,s) dW(s), \qq t\in [0,T],
\end{equation}
which is the {\it backward stochastic Volterra integral equation} (BSVIE, for short) introduced by Yong \cite{Yong2,Yong4} firstly,
motivated by the study of optimal control for forward stochastic Volterra integral equation (FSVIE, for short).
A special case of (\ref{46}) with $f(\cdot)$ independent of $Z(s,t)$ and $\psi(\cd)\equiv \xi$ was studied by Lin \cite{Lin} a little earlier, and followed by several other researchers: Djordjevic and Jankovic \cite{Djordjevic-Jankovic2013,Djordjevic-Jankovic2015}, Hu and Oksendal \cite{Hu-Oksendal2019}.
Since Yong \cite{Yong2,Yong4} introduced BSVIEs of the form \rf{46} (containing $Z(s,t)$), BSVIEs have attracted many researchers' interest.
For example, Anh, Grecksch and Yong \cite{Anh-Grecksch-Yong2010} investigated BSVIEs in Hilbert spaces;
Shi, Wang and Yong \cite{Shi2} studied well-posedness of mean-field BSVIEs;
Overbeck and R\"{o}der \cite{Overbeck-Roder2018} developed a theory of path-dependent BSVIEs;
The numerical aspect was considered by Bender and Pokalyuk \cite{Bender-Pokalyuk2013};
relevant optimal control problems were studied by Shi, Wang and Yong \cite{Shi4},
Agram, Oksendal and Yakhlef \cite{Agram-Oksendal-Yakhlef2016,Agram-Oksendal-Yakhlef2018},
Wang and Zhang \cite{Wang-Zhang2017}, and Wang \cite{Wang2019}; Wang and Yong \cite{Yong} established various comparison theorems for both adapted solutions and adapted M-solutions to BSVIEs in multi-dimensional Euclidean spaces;
Kromer and Overbeck \cite{Kromer-Overbeck2017} introduced a differentiability result for BSVIEs and apply this result to derived continuous-time dynamic capital allocations; Wang, Sun and Yong \cite{Wang-Sun-Yong2018} investigated the well-posedness of quadratic BSVIEs; Wang and Yong \cite{Wang-Yong2018} established a representation of adapted M-solutions to BSVIEs in terms of the solution to a system of (non-classical) partial differential equations and the solution to a (forward) stochastic differential equation, etc.

\ms

Comparing with BSVIE (\ref{46}), we notice that there are two independent Brownian motions $W(\cd)$ and $B(\cd)$ in BDSVIE (\ref{1}). The extra noise $B(\cd)$ in the equation can be seen as some extra information that cannot be detected in practice, such as in a derivative securities market, the inner information is only observable to some particular investors.
The coefficient $g(\cd)$ in the backward It\^{o}'s integral will arise some extra difficulties, due to that there is not any It\^{o}'s formula suitable to deal with BSVIEs so far.

\ms

On the other hand, when $f(\cd)$, $g(\cd)$ and $\psi(\cd)$ are independent of $t$, \rf{1} is reduced to the form
\bel{48}
Y(t)=\xi+\int_t^Tf(s,Y(s),Z(s))ds+\int_t^Tg(s,Y(s),Z(s))d\overleftarrow{B}_{s}-\int_t^TZ(s)dW(s), \qq t\in[0,T],
\ee
where $\xi$ is an $\mathcal{F}_{T}$-measurable square-integrable random variable. \rf{48} is the integral form of so-called {\it backward doubly stochastic differential equation} (BDSDE, for short) which were introduced by Pardoux and Peng \cite{Peng}, to relate a class of systems of quasilinear parabolic backward stochastic PDEs (SPDEs in short).
Since then, many researches have developed the theory and application of BDSDEs. For instance,
the generalized BDSDEs and SPDEs with nonlinear Neumann boundary conditions have been investigated by Boufoussi, Casteren and  Mrhardy \cite{Boufoussi}; Diehl and Friz \cite{Friz} established the connection between BSDEs with rough drivers and BDSDEs; Shi, Gu and Liu obtained the comparison theorem for BDSDEs; Xiong \cite{Xiong-13} obtained the strong uniqueness for the solution to a class of SPDEs; Han, Peng and Wu \cite{Han} investigated the optimal control problems for backward doubly stochastic control systems, and obtained the maximum principle for backward doubly stochastic control systems.
Wen and Shi \cite{Wen-Shi2019} established the relation between BDSDEs with random coefficients and quasilinear stochastic PDEs. A simple glance tells us that BDSVIE (\ref{1}) is a natural generalization of backward stochastic differential equations (BSDEs, for short). See Pardoux and Peng \cite{Peng1}, El Karoui, Peng and Quenez \cite{Peng2}, Ma and Yong \cite{Ma} (and references cited therein) for systematic discussions about BSDEs.

Comparing with BDSDE (\ref{48}), the main features of BDSVIE (\ref{1}) are as follows:
(i) the coefficients $f(\cd)$ and $g(\cd)$ depend on both $t$ and $s$, which implies that (\ref{1}) cannot generally be reduced to a BDSDE (or a BSDE); (ii) $f(\cd)$ and $g(\cd)$ depend on $Z(t,s)$ and $Z(s,t)$; (iii) the free term $\psi(\cd)$ is allowed to be only $\mathcal{B}([0,T])\otimes \mathcal{F}_{T}$-measurable (not necessarily $\mathcal{F}_{t}$-measurable), where $\mathcal{B}([0,T])$ is the Borel $\sigma$-field of $[0,T]$.

\ms

Besides the interest from the mathematical aspect, there are some interesting potential applications motivating the study of BDSVIE (\ref{1}), which has momentous applications in stochastic optimal control, (stochastic) PDEs, mathematical finance and risk management. Let us briefly mention some of them.

\ms

Stochastic optimal control: The classical stochastic optimal control problems mainly focus on (forward) stochastic differential equations (see Yong and Zhou \cite{Yong5}).
In reality, the state equations might involve in memories.
One way to describe such situations is to use stochastic Volterra integral equations (see Shi, Wang and Yong \cite{Shi2,Shi4}, Yong \cite{Yong4}). About this topic, see Section 5 for detail discussion, where
we will elaborate the application of BDSVIEs to stochastic optimal control problems of (forward) doubly stochastic Volterra integral equations (FDSVIEs, for short).

\ms

Stochastic PDEs: It is well-known that quasi-linear PDEs are related to Markovian BSDEs (see Peng \cite{Peng91}, and Pardoux and Peng \cite{Peng92}), what was called the nonlinear Feynman-Kac formula.
Along this way, quasi-linear stochastic PDEs are related to Markovian BDSDEs (see Pardoux and Peng \cite{Peng}), and a system of (non-classical) PDEs are connected to BSVIEs (see Wang and Yong \cite{Wang-Yong2018}). So, it is natural to believe that BDSVIEs could relate to (stochastic) PDEs. In fact, this topic is undergoing (see Shi, Wen and Xiong \cite{Shi-Wen-Xiong2019}).

\ms

Risk measurement: The (static) coherent risk measure was introduced by Artzner-Delbaen-Eber-Heath \cite{Artzner},
and the time-consistent coherent risk measure was introduced by
Peng \cite{Peng2004} and Gianin \cite{Gianin2006}.
When one wants to consider the dynamic version of coherent risk measure and allow possible time-inconsistent preference,
the theory of BSVIEs will play an important role. Some illustrative examples can be found in Yong \cite{Yong3}, and Wang, Sun and Yong \cite{Wang-Sun-Yong2018}.
Similar to BSVIEs, it is natural that BDSVIEs could also be used to describe some general dynamic risk measures. We will discuss this issue in the coming future.

\ms

To our best knowledge, there are few works concerning BDSVIEs.
In this paper, we establish a preliminary theory for BDSVIEs. The novelty mainly contains as follows.
First, well-posedness of BDSVIEs and FDSVIEs in the sense of introduced M-solutions (respectively) was established using the contraction mapping principle. It should be pointed out that the method used to prove the existence and uniqueness theorem (see Theorem \ref{22} below) is more convenient and simpler than the four steps method used in Yong \cite{Yong4}.
Second, as comparison theorem is an important tool and plays an important role in BDSVIEs, a kind of one-dimensional comparison theorem is established for BDSVIEs. As an application of the comparison theorem, we obtain the existence of solutions for BDSVIEs with continuous coefficients.
Finally, we consider the application of BDSVIEs in optimal control problem and prove a duality principle between linear FDSVIEs and BDSVIEs. Moreover, by virtue of the duality principle, a Pontryagin type maximum principle is established for an optimal control problem of FDSVIEs.

\ms

This paper is organized as follows.
In Section 2, we present some preliminary results and lemma which are useful in the sequel.
Section 3 is devoted to the study of the existence and uniqueness of BDSVIEs.
In Section 4, we give a comparison theorem for the solutions of BDSVIEs,
and the existence of solutions for BDSVIEs with continuous coefficients is also given in this section.
The duality principle and Pontryagin type maximum principle are proved in Section 5.

\section{Preliminaries}

Throughout this paper, and recall from the previous section, let $(\Om,\cF,\dbF,\dbP)$ be a complete filtered probability space with $\cF$ being defined in \rf{0}. Let $W=\{W(t); 0\les t<\i\}$ and $B=\{B(t); 0\les t<\i\}$  be two mutually independent standard Brownian motion processes, with values respectively in $\mathbb{R}^{d}$ and in $\mathbb{R}^{l}$, defined on $(\Om,\cF,\dbF,\dbP)$. As usual, we understand equalities and inequalities between random variables in the $\mathbb{P}$-almost sure sense. Let $T>0$ be a fixed terminal time, and $\D$ and $\D^c$ be defined by
$$\Delta=\{ (t,s)\in[0,T]^{2}| \ t\les s\},\qq\Delta^{c}=\{ (t,s)\in[0,T]^{2}| \ t>s\}.$$
The Euclidean norm of a vector $x\in\mathbb{R}^{k}$ will be denoted by $|x|$, and for a $k\times d$ matrix $A$, we define $\| A \|=\sqrt{Tr AA^{\ast}}$.  For any $t\in[0,T]$ and Euclidean space $\dbH$, we introduce the following spaces:
$$\ba{ll}
\ns\ds L^2_{\cF_t}(\Om;\dbH)=\Big\{\xi:\Om\to\dbH\bigm|\xi\hb{ is $\cF_t$-measurable, }\dbE|\xi|^2<\i\Big\},\\
\ns\ds L_{\mathcal{F}_{T}}^2(0,T;\dbH)=\Big \{\psi:\Omega\times [0,T]\rightarrow \dbH \bigm| \psi(\cd) \hb{ is $\mathcal{F}_{T}$-measurable, } \dbE\int_0^{T} |\psi(t)|^{2} dt < \infty\Big\},\\
\ns\ds S_{\mathcal{F}_{T}}^2(0,T;\dbH)=\Big \{\psi:\Omega\times [0,T]\rightarrow \dbH \bigm| \psi(\cd) \hb{ is $\mathcal{F}_{T}$-measurable, continuous, } \dbE\Big(\sup\limits_{0\leq t\leq T}|\psi(t)|^{2}\Big) < \infty\Big\},\\
\ns\ds L_\dbF^2(0,T;\dbH)\1n=\1n\Big\{Y:\Om\1n\times\1n[0,T]\to\dbH\bigm|Y(\cd)\hb{ is
$\dbF$-measurable, }\dbE\int^T_0\1n|Y(s)|^2ds\1n<\2n\i\Big\},\\
\ns\ds S_\dbF^2(0,T;\dbH)\1n=\1n\Big\{Y:\Om\1n\times\1n[0,T]\to\dbH\bigm|Y(\cd)\hb{ is
$\dbF$-measurable, continuous, }\dbE\Big(\sup\limits_{0\leq t\leq T}|Y(t)|^{2}\Big)\1n<\2n\i\Big\},\\
\ns\ds L_{\dbF}^2(\Delta;\dbH)
  =\Big\{Z:\Omega\times \Delta \rightarrow \dbH  \bigm| \hb{For } t\in[0,T], \ Z(t,\cd) \hb{ is $\dbF$-measurable, }\\
\ns\ds\qq\qq\qq\qq\qq\qq\q~ \ \qq\q \ \
\dbE\int_0^{T} \int_t^{T} |Z(t,s)|^{2} dsdt< \infty\Big\},\\
\ns\ds L_{\dbF}^2([0,T]^2;\dbH)
  =\Big\{Z:\Omega\times [0,T]^2 \rightarrow \dbH  \bigm| \hb{For } t\in[0,T], \ Z(t,\cd) \hb{ is $\dbF$-measurable, }\\
\ns\ds\qq\qq\qq\qq\qq\qq\q~ \ \qq\q \ \
\dbE\int_0^{T} \int_0^{T} |Z(t,s)|^{2} dsdt< \infty\Big\}.\\
\ns\ns\ds \mathcal{H}_{\Delta}^{2}[0,T]=L_{\dbF}^2(0,T;\dbH)\times L_{\dbF}^2(\Delta;\dbH);\qq
\mathcal{H}^{2}[0,T]=L_{\dbF}^2(0,T;\dbH)\times L_{\dbF}^2([0,T]^{2};\dbH).
\ea$$
Note that $\mathbb{F}=\{\mathcal{F}_{t},t\in[0,T]\}$ (see \rf{0}) is not a filtration, so if a process $Y(\cd)$ belongs to $L_{\dbF}^2(0,T;\dbH)$, we could say that $Y(\cd)$ is $\dbF$-measurable, but not $\dbF$-adapted. Similarly, we say  a pair of processes is a solution of BDSVIEs, not an adapted solution of BDSVIEs.
\begin{definition} \rm
A pair of processes $(Y(\cdot),Z(\cdot,\cdot))\in \mathcal{H}^{2}[0,T]$ is called a solution of BDSVIE
\rf{1} if \rf{1} is satisfied in the usual It\^{o} sense for Lebesgue measure almost every $t\in[0,T]$.
\end{definition}
\begin{itemize}
  \item [$\mathbf{(H1)}$]
Let
$f:\Omega\times [0,T]\times \mathbb{R}^{k}\times \mathbb{R}^{k\times d}\rightarrow \mathbb{R}^{k},$ and
$g:\Omega\times [0,T]\times \mathbb{R}^{k}\times \mathbb{R}^{k\times d}\rightarrow \mathbb{R}^{k\times l}$
be jointly measurable such that for any $(y,z)\in \mathbb{R}^{k}\times \mathbb{R}^{k\times d}$,
$f(\cdot,y,z)\in L_{\dbF}^{2}(0,T; \mathbb{R}^{k})$ and $g(\cdot,y,z)\in L_{\dbF}^{2}(0,T; \mathbb{R}^{k\times l})$.
 Moreover, there exist some constants $c>0$ and $0<\alpha<1$ such that
 for any $t\in [0,T],(y,z),(y',z')\in \mathbb{R}^{k}\times \mathbb{R}^{k\times d}$,
\begin{equation*}
 \begin{cases}
   |f(t,y,z)-f(t,y',z')|^{2}\les c(|y-y'|^{2} + \|z-z'\|^{2});\\
   \|g(t,y,z)-g(t,y',z')\|^{2}\les c|y-y'|^{2} + \alpha \|z-z'\|^{2}.\\
 \end{cases}
\end{equation*}
\end{itemize}

The following results concerning on BSDEs and BDSDEs are by now well known, for their proofs the reader is referred to Pardoux and Peng \cite{Peng},
 Shi et al. \cite{Shi}, and Lepeltier and Martin \cite{San}.
In detail, our Proposition \ref{3} is Proposition 1.2 of \cite{Peng},
 Proposition \ref{4} is Theorem 3.1 of \cite{Shi}, and Proposition \ref{40} is Lemma 1 of \cite{San}.

\begin{proposition}\label{3}\rm
Let (H1) hold, then for any $\xi\in L^2_{\mathcal{F}_{T}}(\Om;\mathbb{R}^{k})$, BDSDE (\ref{48}) admits a unique solution
$(Y,Z)\in S_{\dbF}^{2}(0,T; \mathbb{R}^{k})\times L_{\dbF}^{2}(0,T; \mathbb{R}^{k\times d}).$
\end{proposition}

\begin{proposition}\label{4}\rm
For $i=1,2,$ given $\xi^i\in L^2_{\mathcal{F}_{T}}(\Om;\mathbb{R}^{k})$ and suppose that $f^i(\cd)$ and $g(\cd)$ satisfy (H1). Let $(Y^{i}(\cd),Z^{i}(\cd,\cd))$ be the solutions of the following equations, respectively for $i=1,2$,
\begin{equation*}
   Y^{i}(t)=\xi^{i} + \int_t^T f^{i}(s,Y^{i}(s),Z^{i}(s)) ds
            + \int_t^T g(s,Y^{i}(s),Z^{i}(s)) d\overleftarrow{B}(s)- \int_t^T Z^{i}(s) dW(s),\qq t\in[0,T].
\end{equation*}
Then, if $\xi^{1}\ges \xi^{2}$ and $f^{1}(t,y,z)\ges f^{2}(t,y,z)$ for every $(t,y,z)\in[0,T]\times \mathbb{R}^{k}\times \mathbb{R}^{k\times d}$, we have
   \begin{equation*}
   Y^{1}(t)\ges Y^{2}(t), \qq t\in[0,T].
   \end{equation*}
\end{proposition}

\begin{proposition}\label{40}\rm
Let $f : \mathbb{R}^{k} \rightarrow \mathbb{R}$ be a continuous function with linear growth, i.e., there exists a positive constant $M$ such that for all $x\in \mathbb{R}^{k}$, $|f(x)|\les M(1+ |x|)$. Then the sequence of functions
\begin{equation*}
  f_{n}(x)\deq\inf_{y\in \mathbb{Q}} \{f(y)+n|x-y|\},
\end{equation*}
is well defined for $n\ges M$, and it satisfies
\begin{itemize}
  \item [(i)]   linear growth: $\forall x\in \mathbb{R}^{k}$, $|f_{n}(x)|\leq M(1+ |x|)$;
  \item [(ii)]  monotonicity in $n$: $\forall x\in \mathbb{R}^{k}$, $f_{n}(x) \nearrow$;
  \item [(iii)] Lipschitz condition: $\forall x,y\in \mathbb{R}^{k}$, $|f_{n}(x)-f_{n}(y)|\leq M|x-y|$;
  \item [(iv)]  strong convergence: if $x_{n} \xrightarrow{ \ n \ } x $, then  $f_{n}(x_{n})\xrightarrow{ \ n \ } f_{n}(x)$.
\end{itemize}
\end{proposition}
At the end of this section, we present a lemma concerning a simple BDSVIE, which is useful in the sequel.
The method used to prove the following lemma, similar to Shi and Wang \cite{Shi1}, is inspired by the method of estimating the adapted solutions of BSDEs in El Karoui and Huang \cite{El}.
\begin{lemma}\label{19}\rm
Suppose that $\psi(\cd)\in L_{\mathcal{F}_{T}}^2(0,T;\mathbb{R}^{k})$,
$f(\cd)\in L_{\dbF}^2(\Delta;\mathbb{R}^{k})$ and $g(\cd)\in L_{\dbF}^2(\Delta;\mathbb{R}^{k\times l})$.
Then BDSVIE:
\begin{equation}\label{2}
   Y(t)=\psi(t)+\int_t^Tf(t,s)ds+\int_t^Tg(t,s)d\overleftarrow{B}(s)-\int_t^TZ(t,s)dW(s), \qq t\in [0,T],
\end{equation}
has a unique solution $(Y(\cdot),Z(\cdot,\cdot))$ $\in \mathcal{H}_{\Delta}^{2}[0,T]$, and for some constant $\b>0,$ the following estimate holds,
\bel{20}\ba{ll}
\ds\dbE\int_0^T e^{\beta t} |Y(t)|^{2} dt  +  \dbE\int_0^T\int_t^T  e^{\beta s} \|Z(t,s)\|^{2} dsdt \\
\ns\ds\les4\dbE\int_0^Te^{\beta t}|\psi(t)|^{2}dt+\frac{10}{\beta}\dbE\int_0^T\int_t^T e^{\beta s}|f(t,s)|^{2}dsdt \\
\ns\ds\q+ \dbE\int_0^T e^{\beta t}\int_t^T \|g(t,s)\|^{2} ds dt+\dbE\int_0^T \int_t^T e^{\beta s}\|g(t,s)\|^{2}dsdt.
\ea\ee
\end{lemma}

\begin{proof}
We introduce the following family of BDSDEs (parameterized by $t\in [0,T]$):
\begin{equation}\label{5}
   \lambda(t,r)=\psi(t) + \int_r^T f(t,s) ds + \int_r^T g(t,s) d\overleftarrow{B}(s) - \int_r^T \mu(t,s) dW(s), \qq r\in [t,T].
\end{equation}
From Proposition \ref{3}, Eq. (\ref{5}) admits a unique solution
$(\lambda(t,\cdot),\mu(t,\cdot))$ on $[t,T]$,  for every $t\in [0,T]$.
Let
\begin{equation*}
Y(t)=\lambda(t,t), \qq Z(t,s)=\mu(t,s), \qq 0\leq t\leq s\leq T.
\end{equation*}
Then $(Y(\cdot),Z(\cdot,\cdot))$ is a solution to (\ref{2}), which implies the existence of BDSVIE \rf{2}.
From (\ref{5}), we have
\bel{42}
\lambda(t,r)+\int_r^T \mu(t,s) dW(s)=\psi(t) + \int_r^T f(t,s) ds + \int_r^T g(t,s) d\overleftarrow{B}(s), \qq r\in [t,T].
\ee
Especially, when $r=t$, we obtain that
\begin{equation}\label{57}
 Y(t)+ \int_t^T Z(t,s) dW(s)=\psi(t) + \int_t^T f(t,s) ds + \int_t^T g(t,s) d\overleftarrow{B}(s),\qq t\in[0,T].
\end{equation}
Before proving the uniqueness of BDSVIE \rf{2}, we estimate
\begin{equation*}
   \dbE\int_0^T e^{\beta t} |Y(t)|^{2} dt + \dbE\int_0^T \int_t^T e^{\beta s} \|Z(t,s)\|^{2} dsdt.
\end{equation*}
By Cauchy-Schwarz inequality, it follows that
\begin{align}
 \bigg|\int_s^T f(t,u) du\bigg|^{2}
  =& \bigg|\int_s^T e^{\frac{-\gamma u}{2}} e^{\frac{\gamma u}{2}} f(t,u) du\bigg|^{2}  \nonumber\\
  \les& \int_s^T e^{-\gamma u}  du \cdot \int_s^T e^{\gamma u} |f(t,u)|^{2} du
  \les \frac{1}{\gamma} e^{-\gamma s} \int_s^T e^{\gamma u} |f(t,u)|^{2} du, \qq s\in[t,T],  \label{6}
\end{align}
where $\gamma=\frac{\beta}{2}$ or $\beta$. By taking $\gamma=\frac{\beta}{2}$ in (\ref{6}), we see that
\begin{equation*}
\begin{split}
    \int_t^T \beta e^{\beta s}  \bigg|\int_s^T f(t,u) du\bigg|^{2} ds
\les& \frac{4}{\beta} \int_t^T \frac{\beta}{2} e^{\frac{\beta}{2} s}
        \cdot\bigg(\int_s^T e^{\frac{\beta}{2} u} |f(t,u)|^{2} du\bigg) ds\\
\les& \bigg(\frac{4}{\beta}e^{\frac{\beta}{2} s} \int_s^T e^{\frac{\beta}{2} u} |f(t,u)|^{2} du\bigg) \bigg|_{t}^{T}
        +\frac{4}{\beta} \int_t^T e^{\beta s} |f(t,s)|^{2} ds\\
\les& \frac{4}{\beta} \int_t^T e^{\beta s} |f(t,s)|^{2} ds.
\end{split}
\end{equation*}
Therefore
\begin{equation}\label{7}
  \dbE\int_0^T \int_t^T \beta e^{\beta s} \bigg|\int_s^T f(t,u) du\bigg|^{2} dsdt
  \les \frac{4}{\beta}  \dbE\int_0^T \int_t^T e^{\beta s} |f(t,s)|^{2} dsdt.
\end{equation}
We also obtain the following result by taking $s=t$ and $\gamma=\beta$ in (\ref{6}),
\begin{equation}\label{8}
  \dbE\int_0^T e^{\beta t} \bigg|\int_t^T f(t,s) ds\bigg|^{2} dt
  \les \frac{1}{\beta}  \dbE\int_0^T \int_t^T e^{\beta s} |f(t,s)|^{2} dsdt.
\end{equation}
For the generator $g(\cd)$, since
\begin{equation*}
\begin{split}
       \dbE\int_t^T \beta e^{\beta s}\bigg(\int_s^T \|g(t,u)\|^{2} du\bigg) ds
     =& \ \dbE\bigg( e^{\beta s} \int_s^T \|g(t,u)\|^{2}du \bigg) \bigg|_{t}^{T}
     + \dbE \int_t^T e^{\beta s} \|g(t,s)\|^{2} ds\\
  \les& \ \dbE \int_t^T e^{\beta s} \|g(t,s)\|^{2} ds,
\end{split}
\end{equation*}
we obtain
\begin{equation}\label{10}
  \dbE\int_0^T \int_t^T \beta e^{\beta s}\bigg(\int_s^T \|g(t,u)\|^{2} du\bigg) dsdt
  \les  \dbE\int_0^T \int_t^T e^{\beta s} \|g(t,s)\|^{2} dsdt.
\end{equation}
For the solution $Z(\cd,\cd)$, it's easy to see that
\begin{equation}\label{11}
    \int_r^T \beta e^{\beta s}\bigg(\int_s^T \|Z(t,u)\|^{2} du\bigg) ds
   =\bigg(e^{\beta s}\int_s^T \|Z(t,u)\|^{2} du\bigg)\bigg|_{r}^{T} + \int_r^T  e^{\beta s} \|Z(t,s)\|^{2} ds, \q r\in [t,T].
\end{equation}
For every $ t\in [0,T]$, we can rewrite (\ref{11}) after taking $r=t$,
\bel{12}\dbE\int_0^T \int_t^T  e^{\beta s}\|Z(t,s)\|^{2} dsdt  =\dbE\int_0^T\int_t^T \beta e^{\beta s}\bigg(\int_s^T \|Z(t,u)\|^{2} du\bigg) dsdt
   + \dbE\int_0^T e^{\beta t}\int_t^T \|Z(t,u)\|^{2} du dt.  \ee
Notice $\psi(t)$ is $\mathcal{F}_{T}$-measurable, using the property of conditional expectation,
it follows from (\ref{57}) that
\begin{equation*}
\begin{split}
     \dbE|Y(t)|^{2}+ \dbE\int_t^T \|Z(t,s)\|^{2} ds
   =&\ \dbE\bigg(\psi(t) + \int_t^T f(t,s) ds + \int_t^T g(t,s) d\overleftarrow{B}_{s}\bigg)^{2}\\
\les&\ \dbE\bigg(2|\psi(t)|^{2} + 2\bigg|\int_t^T f(t,s) ds\bigg|^{2} + \int_t^T \|g(t,s)\|^{2} ds\bigg).
\end{split}
\end{equation*}
Then, note that (\ref{8}), we have
\bel{15}\ba{ll}
\ds  \dbE\int_0^T e^{\beta t}|Y(t)|^{2}dt + \dbE\int_0^T e^{\beta t}\int_t^T \|Z(t,s)\|^{2} dsdt\\
\ns\ds\les\dbE \int_0^T\bigg(2 e^{\beta t}|\psi(t)|^{2} + \frac{2}{\beta}\int_t^T e^{\beta s} |f(t,s)|^{2} ds
+ e^{\beta t} \int_t^T \|g(t,s)\|^{2} ds\bigg)dt.
\ea\ee
Similarly, from (\ref{42}), we obtain
\begin{equation*}
     \dbE|\lambda(t,s)|^{2}+ \dbE\int_s^T \|Z(t,u)\|^{2} du
\les \dbE\bigg(2|\psi(t)|^{2} + 2\bigg|\int_s^T f(t,u) du\bigg|^{2} + \int_s^T \|g(t,u)\|^{2} du\bigg).
\end{equation*}
Combining (\ref{7}) and (\ref{10}) implies that
\bel{14}\ba{ll}
\ds  \dbE\int_0^T\int_t^T \beta e^{\beta s}\bigg(\int_s^T \|Z(t,u)\|^{2} du\bigg) dsdt \\
\ns\ds \les\dbE\int_0^T\bigg( 2e^{\beta t}|\psi(t)|^{2} + \frac{8}{\beta}\int_t^T e^{\beta s}|f(t,s)|^{2} ds
       + \int_t^T e^{\beta s}\|g(t,s)\|^{2} ds \bigg)dt.
\ea\ee
Then, the estimate (\ref{20}) holds by combining (\ref{12}), (\ref{15}) and  (\ref{14}). Finally, the uniqueness of (\ref{2}) comes from the estimate (\ref{20}).
\end{proof}

\section{Well-posedness of BDSVIE}
In this section, we prove the existence and uniqueness of BDSVIEs and FDSVIEs.
Since FDSVIE can be regarded as  a ``forward'' type of BDSVIE, we first consider BDSVIE.

\subsection{BDSVIE}

Now we study the existence of uniqueness of BDSVIE \rf{1}, and for simplicity of presentation, we rewrite it as follows:
\bel{55}\ba{ll}
\ds Y(t)=\psi(t) + \int_t^T f(t,s,Y(s),Z(t,s),Z(s,t)) ds\\
\ns\ds\qq\q\ + \int_t^T g(t,s,Y(s),Z(t,s),Z(s,t)) d\overleftarrow{B}(s)- \int_t^T Z(t,s) dW(s), \qq t\in [0,T],
\ea\ee
where
$f(\omega,t,s,y,z,\zeta):\Omega\times \Delta\times \mathbb{R}^{k}\times \mathbb{R}^{k\times d}\times \mathbb{R}^{k\times d}\rightarrow \mathbb{R}^{k}$
and
$g(\omega,t,s,y,z,\zeta):\Omega\times \Delta\times \mathbb{R}^{k}\times \mathbb{R}^{k\times d}\times \mathbb{R}^{k\times d}\rightarrow \mathbb{R}^{k\times l}$
are $\mathcal{F}_{s}$-measurable,
and $\psi(\omega,t):\Omega\times [0,T]\rightarrow \mathbb{R}^{k}$ is $\mathcal{F}_{T}$-measurable.
The purpose is to prove BDSVIE (\ref{55}) admits a unique solution $(Y(\cdot),Z(\cdot,\cdot))$ in $\mathcal{H}^{2}[0,T]$.
However, as showed in Yong \cite{Yong4}, for the sake of the uniqueness of  solutions, some additional constraints should be imposed on $Z(t,s)$ for $(t,s)\in \Delta^{c}$.
In order to do this, similar to Yong \cite{Yong4}, we introduce the M-solution in the circumstance of BDSVIEs.

\ms

By the definition of $\mathbb{F}=\{\mathcal{F}_{t},t\in[0,T]\}$ (see \rf{0}), we see that it is neither increasing nor decreasing, and so it does not constitute a filtration. In order to define the M-solution, we define the filtration $(\mathcal{G}_{t})_{0\leq t\leq T}$ by
$$\mathcal{G}_{t}\deq\mathcal{F}^{W}_{t}\vee \mathcal{F}^{B}_{0,T}.$$
For any fixed $t\in[0,T]$, suppose a process $Y(t)$ is $\mathcal{G}_{t}$-square integrable.
Then $M(\cd)$ defined by
\begin{equation*}
  M(r) = \dbE [Y(t)|\mathcal{G}_{r}], \qq r\in[0,t],
\end{equation*}
is a $\mathcal{G}_{r}$-square integrable martingale.
An obvious extension of It\^{o}'s martingale representation theorem yields that there exists a unique $\mathcal{G}_{r}$-progressively measurable process $Z(t,\cd)$ (parameterized by $t\in[0,T]$) such that
\begin{equation*}
  M(r) = M(0) + \int_0^r Z(t,s)dW(s), \qq r\in[0,t].
\end{equation*}
In particular, when $r=t$, we have
\begin{equation*}
  Y(t) =\dbE[Y(t)] + \int_0^t Z(t,s)dW(s), \qq  t\in[0,T].
\end{equation*}
It should be pointed out that in the above equation, if $Y(\cd)$ is in fact $\dbF$-measurable, then $Z(t,\cd)$ is also $\dbF$-measurable (see Pardoux and Peng \cite{Peng}). Owing to the above idea, we define the M-solution of BDSVIE (\ref{55}) as follows:
\begin{definition} \rm
For any $S\in[0,T)$, a pair of processes $(Y(\cdot),Z(\cdot,\cdot))\in \mathcal{H}^{2}[0,T]$ is called an
M-solution of BDSVIE (\ref{55}) if  (\ref{55}) is satisfied in the usual It\^{o}'s sense for Lebesgue measure almost
every $t\in[S,T]$ and, in addition, the following relation holds:
\bel{54}Y(t) =\dbE[Y(t)|\cF_S]+\int_S^tZ(t,s)dW(s),\q a.e. \ t\in[S,T].\ee
\end{definition}
Let $\mathcal{M}^{2}[0,T]$ be the set of all pair of $(Y(\cdot),Z(\cdot,\cdot))\in \mathcal{H}^{2}[0,T]$ such that (\ref{54}) holds. Obviously $\mathcal{M}^{2}[0,T]$ is a closed subspace of $\mathcal{H}^{2}[0,T]$.
Note that for any $(Y(\cdot),Z(\cdot,\cdot))\in \mathcal{M}^{2}[0,T]$, one can show that
\begin{equation*}
        \dbE\int_0^{T} \bigg(e^{\beta t}|Y(t)|^{2} + \int_0^{T} e^{\beta s}\|Z(t,s)\|^{2} ds\bigg) dt
  \les 2\dbE\int_0^{T} \bigg(e^{\beta t}|Y(t)|^{2} + \int_t^{T} e^{\beta s}\|Z(t,s)\|^{2} ds\bigg) dt.
\end{equation*}
This means that we can use the following as an equivalent norm in $\mathcal{M}^{2}[0,T]$:
 \begin{equation*}
 \| (Y(\cdot),Z(\cdot,\cdot)) \|_{\mathcal{M}^{2}[0,T]}
 \equiv \bigg[ \dbE\int_0^{T} \bigg(e^{\beta t}|Y(t)|^{2} + \int_t^{T} e^{\beta s}\|Z(t,s)\|^{2} ds\bigg) dt \bigg]^{\frac{1}{2}}.
 \end{equation*}
\begin{remark} \rm
Under the assumptions of Lemma  \ref{19}, if we define $Z(\cdot,\cdot)$ on $\Delta^{c}$ by the relation (\ref{54}),
then BDSVIE (\ref{2}) admits a unique M-solution in $\mathcal{H}^{2}[0,T]$.
\end{remark}
\begin{itemize}
  \item [$\mathbf{(H2)}$]
  Assume that
$f:\Omega\times \Delta\times \mathbb{R}^{k}\times \mathbb{R}^{k\times d}\times \mathbb{R}^{k\times d}\rightarrow \mathbb{R}^{k}$ and
$g:\Omega\times \Delta\times \mathbb{R}^{k}\times \mathbb{R}^{k\times d}\times \mathbb{R}^{k\times d}\rightarrow \mathbb{R}^{k\times l}$
are jointly measurable such that for all $(y,z,\zeta)\in \mathbb{R}^{k}\times \mathbb{R}^{k\times d}\times \mathbb{R}^{k\times d}$,
$f(\cdot,\cdot,y,z,\zeta)\in L_{\dbF}^{2}(\Delta; \mathbb{R}^{k})$ and
$g(\cdot,\cdot,y,z,\zeta)\in L_{\dbF}^{2}(\Delta; \mathbb{R}^{k\times l})$.
Furthermore, there exist some constants $c>0$ and $0<\alpha<\frac{1}{T+2}$ such that for any  $y,y'\in \mathbb{R}^{k},z,z',\zeta,\zeta' \in {R}^{k\times d}$ and
$(t,s) \in \Delta$,
\begin{equation*}
 \begin{cases}
   |f(t,s,y,z,\zeta)-f(t,s,y',z',\zeta')|^{2}\les c(|y-y'|^{2} + \|z-z'\|^{2} + \|\zeta-\zeta'\|^{2});\\
   \|g(t,s,y,z,\zeta)-g(t,s,y',z',\zeta')\|^{2}\les \alpha(|y-y'|^{2} + \|z-z'\|^{2} + \|\zeta-\zeta'\|^{2}).\\
 \end{cases}
\end{equation*}
\end{itemize}
Now we prove the existence and uniqueness of BDSVIE \rf{55}.
For notational simplicity, we denote $f_{0}(t,s)=f(t,s,0,0,0)$ and $g_{0}(t,s)=g(t,s,0,0,0)$.
\begin{theorem}\label{22} \sl
  Under the assumption (H2), for any  $\psi(\cdot)\in L_{\mathcal{F}_{T}}^2(0,T;\mathbb{R}^{k})$,
  BDSVIE (\ref{55}) admits a unique M-solution $(Y(\cdot),Z(\cdot,\cdot))\in \mathcal{H}^{2}[0,T]$.
  Moreover, the following estimate holds,
\bel{18}\ba{ll}
\ds\dbE\bigg[\int_0^T e^{\beta t} |Y(t)|^{2} dt + \int_0^T \int_t^T e^{\beta s} \|Z(t,s)\|^{2} dsdt\bigg] \\
\ns\ds\les L\dbE\bigg[\int_0^T |\psi(t)|^{2} dt +  \int_0^T \int_t^T e^{\beta s}|f_{0}(t,s)|^{2} dsdt
         +\int_0^T \int_t^T e^{\beta s}\|g_{0}(t,s)\|^{2} dsdt\bigg],
\ea\ee
where $L$ is a positive constant which may be different from line to line.
\end{theorem}

\begin{proof}
For any $(y(\cdot),z(\cdot,\cdot))\in \mathcal{M}^{2}[0,T]$, consider the following BDSVIE,
\bel{23}\ba{ll}
\ns\ds Y(t)=\psi(t) + \int_t^T f(t,s,y(s),z(t,s),z(s,t)) ds\\
\ns\ds\qq\q+ \int_t^T g(t,s,y(s),z(t,s),z(s,t)) d\overleftarrow{B}(s) - \int_t^T Z(t,s) dW(s), \qq t\in[0,T].
\ea\ee
From Lemma \ref{19}, there exists a unique solution $(Y(\cdot),Z(\cdot,\cdot))\in \mathcal{H}_{\Delta}^{2}[0,T]$.
Now if we define $Z(\cdot,\cdot)$ on $\Delta^{c}$ by the relation (\ref{54}),
then $(Y(\cdot),Z(\cdot,\cdot))\in \mathcal{M}^{2}[0,T]$ is an M-solution to (\ref{23}).
Thus, the map $\Gamma:(y(\cdot),z(\cdot,\cdot))\mapsto (Y(\cdot),Z(\cdot,\cdot))$ is well-defined. Next, we prove $\G$ is a contraction on $\mathcal{M}^{2}[0,T]$.
By the estimate (\ref{20}), we have
\begin{align}
\dbE&\int_0^T e^{\beta t} |Y(t)|^{2} dt + \dbE\int_0^T \int_t^T e^{\beta s} \|Z(t,s)\|^{2} dsdt \nonumber\\
   \les&\ 4\dbE\int_0^Te^{\beta t} |\psi(t)|^{2} dt
        + \frac{10}{\beta} \dbE\int_0^T \int_t^T e^{\beta s}|f(t,s,y(s),z(t,s),z(s,t))|^{2} dsdt \label{24}\\
       &+ \dbE\int_0^T e^{\beta t}\int_t^T \|g(t,s,y(s),z(t,s),z(s,t))\|^{2} dsdt \label{26}\\
       &+ \dbE\int_0^T \int_t^T e^{\beta s} \|g(t,s,y(s),z(t,s),z(s,t))\|^{2} dsdt. \label{27}
\end{align}
For the second term in (\ref{24}), from (H2) one has
\begin{equation*}
\begin{split}
       \dbE& \int_0^T \int_t^T e^{\beta s} |f(t,s,y(s),z(t,s),z(s,t))|^{2} dsdt\\
  \les&\ \dbE \int_0^T \int_t^T e^{\beta s}\left(|f_{0}(t,s)|^{2} + c|y(s)|^{2} + c \|z(t,s)\|^{2} + c \|z(s,t)\|^{2}\right) dsdt\\
  \les&\  c(T+1)\dbE \int_0^T e^{\beta t}|y(t)|^{2} dt
        + c\dbE\int_0^T \int_t^T e^{\beta s} \|z(t,s)\|^{2} ds dt\\
       & +\dbE \int_0^T \int_t^T e^{\beta s}|f_{0}(t,s)|^{2} dsdt.
\end{split}
\end{equation*}
Similarly, for the term (\ref{26}), one has
\begin{equation*}
\begin{split}
       \dbE&\int_0^T e^{\beta t}\int_t^T \|g(t,s,y(s),z(t,s),z(s,t))\|^{2} dsdt\\
  \les&\ \dbE \int_0^T e^{\beta t} \int_t^T\left(\|g_{0}(t,s)\|^{2}
         + \alpha|y(s)|^{2} + \alpha \|z(t,s)\|^{2} + \alpha \|z(s,t)\|^{2}\right) dsdt\\
  \les&\ \big(\frac{\alpha}{\beta}+\alpha \big) \dbE \int_0^T e^{\beta t}|y(t)|^{2} dt
        + \alpha \dbE\int_0^T \int_t^T e^{\beta s}\|z(t,s)\|^{2} ds dt\\
       & + \dbE \int_0^T \int_t^T e^{\beta s}\|g_{0}(t,s)\|^{2} dsdt.
\end{split}
\end{equation*}
Also, for the term (\ref{27}), one has
\begin{equation*}
\begin{split}
       \dbE& \int_0^T \int_t^T e^{\beta s} \|g(t,s,y(s),z(t,s),z(s,t))\|^{2} dsdt\\
  \les&\ \dbE\int_0^T \int_t^T e^{\beta s}\left(\|g_{0}(t,s)\|^{2} + \alpha|y(s)|^{2}
           + \alpha \|z(t,s)\|^{2} + \alpha \|z(s,t)\|^{2}\right) dsdt\\
  \les&\     \alpha(T+1)\dbE \int_0^T e^{\beta t}|y(t)|^{2} dt
           + \alpha \dbE\int_0^T \int_t^T e^{\beta s}\|z(t,s)\|^{2} ds dt\\
      &    + \dbE \int_0^T \int_t^T e^{\beta s}\|g_{0}(t,s)\|^{2} dsdt.
\end{split}
\end{equation*}
Hence we deduce
\bel{25}\ba{ll}
\ds\dbE\int_0^T e^{\beta t} |Y(t)|^{2} dt + \dbE\int_0^T \int_t^T e^{\beta s} \|Z(t,s)\|^{2} dsdt \\
\ns\ds\les4\dbE\int_0^Te^{\b t}|\psi(t)|^{2}dt+\frac{10}{\beta}\dbE\int_0^T\int_t^T e^{\b s}|f_{0}(t,s)|^{2}dsdt \\
\ns\ds\q + 2 \dbE \int_0^T \int_t^T e^{\beta s}\|g_{0}(t,s)\|^{2} dsdt
         + \big[\frac{K}{\beta}+\alpha(T+2)\big]\dbE \int_0^T e^{\beta t}|y(t)|^{2} dt \\
\ns\ds\q + \big(\frac{10c}{\beta}+2\alpha \big)\dbE\int_0^T \int_t^T e^{\beta s}\|z(t,s)\|^{2} ds dt,
\ea\ee
where $K=10c(T+1)+\alpha$.
Then, if $(Y_{i}(\cdot),Z_{i}(\cdot,\cdot))$ is the corresponding M-solution of $(y_{i}(\cdot),z_{i}(\cdot,\cdot))$ to BDSVIE (\ref{23}) with $i=1,2$, we have
$$\ba{ll}
\ds\dbE \bigg(\int_0^{T}e^{\beta t}|Y_{1}(t)-Y_{2}(t)|^{2} dt
   + \int_0^{T}\int_t^{T} e^{\beta s}\|Z_{1}(t,s)-Z_{2}(t,s)\|^{2} ds dt\bigg)\\
\ns\ds \les\delta\dbE \bigg(\int_0^{T}e^{\beta t}|y_{1}(t)-y_{2}(t)|^{2} dt + \int_0^{T}\int_t^{T} e^{\beta s}\|z_{1}(t,s)-z_{2}(t,s)\|^{2} ds dt\bigg),
\ea$$
where $\delta=\frac{K}{\beta}+\alpha(T+2)$. Note that $\alpha<\frac{1}{T+2}$, if we let $\beta$ be some proper constant such that $\beta>\frac{1-\alpha(T+2)}{K}$, then $(y(\cdot),z(\cdot,\cdot))\mapsto (Y(\cdot),Z(\cdot,\cdot))$ is a contraction on $\mathcal{M}^{2}[0,T]$. By the contracting mapping principle, we see that BDSVIE \rf{55} admits a unique M-solution. Finally, the estimate (\ref{18}) directly follows from (\ref{25}). This completes the proof.
\end{proof}

\begin{remark}\rm
One may notice that, in assumption (H2), $\a$ is restricted to the interval $(0,\frac{1}{T+2})$. This is a technical problem caused by the absence of the It\^{o} formula in the theory of BSVIEs. Similar to Pardoux and Peng \cite{Peng}, we hope to relax the restriction of $\a$ to $(0,1)$ in the near future.
\end{remark}
From Theorem \ref{22}, we see that in BDSVIE \rf{55}, if the coefficients $f(\cdot)$ and $g(\cdot)$ are independent of $Z(s,t)$, then the related equation admits a unique solution in $\mathcal{H}_{\Delta}^{2}[0,T]$.
\begin{corollary}\rm
 Suppose that $f=f(t,s,y,z)$ and $g=g(t,s,y,z)$ satisfy (H2). Then for any given terminal condition $\psi(\cdot)\in L_{\mathcal{F}_{T}}^2(0,T;\mathbb{R}^{k})$, the following BDSVIE:
$$Y(t)=\psi(t)+\int_t^Tf(t,s,Y(s),Z(t,s))ds+\int_t^Tg(t,s,Y(s),Z(t,s))d\overleftarrow{B}(s)-\int_t^T Z(t,s) dW(s),
\q~ t\in[0,T],$$
admits a unique solution $(Y(\cdot),Z(\cdot,\cdot))\in \mathcal{H}_{\Delta}^{2}[0,T]$.
\end{corollary}

\subsection{FDSVIE}
In this subsection, due to that FDSVIE can be regarded as a ``forward'' type of BDSVIE,
we briefly show the existence and uniqueness result for FBSVIE. Consider the following FBSVIE:
\bel{61}\ba{ll}
\ds  P(t)=\varphi(t) + \int_0^t b(t,s,P(s),Q(t,s),Q(s,t)) ds\\
\ns\ds\qq\q + \int_0^t \sigma(t,s,P(s),Q(t,s),Q(s,t))dW_{s} - \int_0^t Q(t,s) d\overleftarrow{B}_{s}, \qq t\in [0,T],
\ea\ee
where
$b(\omega,t,s,p,q,\vartheta):\Omega\times \Delta^{c}\times \mathbb{R}^{k}\times \mathbb{R}^{k\times l}\times \mathbb{R}^{k\times l}\rightarrow \mathbb{R}^{k}$
and
$\sigma(\omega,t,s,p,q,\vartheta):\Omega\times \Delta^{c}\times \mathbb{R}^{k}\times \mathbb{R}^{k\times l}\times \mathbb{R}^{k\times l}\rightarrow \mathbb{R}^{k\times d}$
are $\mathcal{F}_{s}$-measurable given maps,
and $\varphi(\omega,t):\Omega\times [0,T]\rightarrow \mathbb{R}^{k}$ is $\mathcal{F}_{t}$-measurable.

\ms

The same as above subsection, in order to obtain the uniqueness of solutions,
we introduce a ``backward'' M-solution for FDSVIE (\ref{61}).
\begin{definition} \rm
A pair of processes $(P(\cdot),Q(\cdot,\cdot))\in \mathcal{H}^{2}[0,T]$ is called a solution of FDSVIE \rf{61} if \rf{61} is satisfied in the usual It\^{o} sense for Lebesgue measure almost every $t\in[0,T]$. In addition, if for any $S\in(0,T]$,
the following relation holds:
\begin{equation}\label{56}
  P(t) =\dbE[P(t)|\cF_S] + \int_t^S Q(t,s)d\overleftarrow{B}(s), \q a.e. \ t\in[0,S],
\end{equation}
then $(P(\cdot),Q(\cdot,\cdot))$ is called an M-solution of FDSVIE (\ref{61}).
\end{definition}
The above definition is based on the ideas of Pardoux and Peng \cite{Peng} and the ``backward" martingale representation theorem.
Let $\mathcal{N}^{2}[0,T]$ be the set of all $(P(\cdot),Q(\cdot,\cdot))\in \mathcal{H}^{2}[0,T]$ such that the relation (\ref{56}) holds. Obviously, $\mathcal{N}^{2}[0,T]$ is a closed subspace of $\mathcal{H}^{2}[0,T]$.
 To guarantee the existence and uniqueness of (\ref{61}), we make the following assumption.
\begin{itemize}
  \item [$\mathbf{(H3)}$]
  Assume $b:\Omega\times \Delta^{c}\times \mathbb{R}^{k}\times \mathbb{R}^{k\times l}\times \mathbb{R}^{k\times l}\rightarrow \mathbb{R}^{k}$ and $\sigma:\Omega\times \Delta^{c}\times \mathbb{R}^{k}\times \mathbb{R}^{k\times l}\times \mathbb{R}^{k\times l}\rightarrow \mathbb{R}^{k\times d}$ are jointly measurable such that
for all $(p,q,\vartheta)\in \mathbb{R}^{k}\times \mathbb{R}^{k\times l}\times \mathbb{R}^{k\times l}$,
$b(\cdot,\cdot,p,q,\vartheta)\in L_{\dbF}^{2}(\Delta^{c}; \mathbb{R}^{k})$ and
$\sigma(\cdot,\cdot,p,q,\vartheta)\in L_{\dbF}^{2}(\Delta^{c}; \mathbb{R}^{k\times d})$.
Furthermore, there exist some constants $c>0$ and $0<\alpha<\frac{1}{T+2}$ such that for any $p,p'\in \mathbb{R}^{k},q,q',\vartheta,\vartheta' \in {R}^{k\times l}$ and $(t,s) \in \Delta^{c}$,
\begin{equation*}
 \begin{cases}
   |b(t,s,p,q,\vartheta)-b(t,s,p',q',\vartheta')|^{2}\les c(|p-p'|^{2} + \|q-q'\|^{2} + \|\vartheta-\vartheta'\|^{2});\\
   \|\sigma(t,s,p,q,\vartheta)-\sigma(t,s,p,q,\vartheta)\|^{2}\les \alpha(|p-p'|^{2} + \|q-q'\|^{2} + \|\vartheta-\vartheta'\|^{2}).
 \end{cases}
\end{equation*}
\end{itemize}
 Similar to BDSVIEs, we have the following existence and uniqueness concerning FDSVIEs. Since the proof of the following Theorem \ref{63} is similar to the proof of Theorem \ref{22}, we only present the result and nevertheless include a complete proof for the convenience of the reader.
%
\begin{theorem}\label{63} \sl
  Let the assumption (H3) hold. Then for any  $\varphi(\cdot)\in L_{\dbF}^2(0,T;\mathbb{R}^{k})$,
  FDSVIE (\ref{61}) admits a unique M-solution $(p(\cdot),q(\cdot,\cdot))\in \mathcal{H}^{2}[0,T]$.
  Moreover, the following estimate holds,
\bel{66}\ba{ll}
\ds \dbE\bigg[\int_0^T e^{\beta t} |p(t)|^{2} dt + \int_0^T \int_0^t e^{\beta s} \|q(t,s)\|^{2} dsdt\bigg] \\
\ns\ds\les L \dbE\bigg[\int_0^T |\varphi(t)|^{2} dt +  \int_0^T \int_0^t e^{\beta s}|b_{0}(t,s)|^{2} dsdt
        +\int_0^T \int_0^t e^{\beta s}\|\sigma_{0}(t,s)\|^{2} dsdt\bigg],
\ea\ee
where $b_{0}(t,s)=b(t,s,0,0,0)$ and $\sigma_{0}(t,s)=\sigma(t,s,0,0,0)$, and $L$ is a constant.
\end{theorem}

\section{Comparison theorem of BDSVIE}

In this section, we consider one-dimensional BDSVIE for which we prove a comparison theorem.
As one of its applications, the existence of solutions of one-dimensional BDSVIE with continuous coefficient is established.

\subsection{Comparison theorem}

For $i=1,2$, we study a comparison theorem of BDSVIE of the following type:
\begin{equation}\label{28}
\begin{split}
   Y^{i}(t)=&\psi^{i}(t) + \int_t^T f^{i}(t,s,Y^{i}(s),Z^{i}(t,s)) ds\\
            &+ \int_t^T g(t,s,Y^{i}(s),Z^{i}(t,s)) d\overleftarrow{B}(s) - \int_t^T Z^{i}(t,s) dW(s), \qq t\in[0,T].
\end{split}
\end{equation}
The key feature here is that the coefficients $f^{i}(\cdot)$ and $g(\cdot)$ are independent of $Z^{i}(s,t)$.
In this situation, for any solution $(Y^{i}(\cdot),Z^{i}(\cdot,\cdot))\in \mathcal{H}_{\Delta}^{2}[0,T]$ of the above equation, $Z^{i}(t,s)$ is only (uniquely) defined on $(t,s)\in \Delta$, and we do not need the values $Z^{i}(t,s)$ of $Z^{i}(\cdot,\cdot)$ for $(t,s)\in \Delta^{c}$. Suppose $f^{i}=f^{i}(t,s,y,z)$ and $g=g(t,s,y,z)$ satisfy (H2).
Then for any $\psi^{i}(\cdot)\in L_{\mathcal{F}_{T}}^2(0,T;\mathbb{R}^{k})$, Eq. (\ref{28}) has a unique solution
 $(Y^{i}(\cdot),Z^{i}(\cdot,\cdot))$ $\in \mathcal{H}_{\Delta}^{2}[0,T]$.

\ms

Next, we want to show a proper comparison between $Y^{1}(\cdot)$ and $Y^{2}(\cdot)$. Here, we only consider one-dimensional case, i.e., $k=l=1$, and for notational convenience, we assume $d=1$.
To begin with, let us consider the following simple BDSVIE: for $i=1,2,$
\begin{equation}\label{29}
   Y^{i}(t)=\psi^{i}(t) + \int_t^T f^{i}(t,s,Z^{i}(t,s)) ds
            +\int_t^T g(t,s,Z^{i}(t,s)) d\overleftarrow{B}(s) - \int_t^T Z^{i}(t,s) dW(s), \qq t\in [0,T].
\end{equation}

\begin{proposition}\label{35}\rm
For $i=1,2,$ let $f^{i}=f^{i}(t,s,z)$ and $g=g(t,s,z)$ satisfy (H2).
Moreover,
$$f^{1}(t,s,z)\ges f^{2}(t,s,z), \q \forall (t,z)\in [0,s]\times \mathbb{R}, \ \ a.s., \  a.e. \ s\in[0,T].$$
Then for any $\psi^{i}(\cdot)\in L_{\mathcal{F}_{T}}^2(0,T;\mathbb{R})$ with
   $\psi^{1}(t)\ges \psi^{2}(t),  \ a.s., \ t\in[0,T],$
the corresponding unique solution $(Y^{i}(\cdot),Z^{i}(\cdot,\cdot))\in \mathcal{H}_{\Delta}^{2}[0,T]$ of (\ref{29}) satisfy
   \begin{equation}\label{32}
   Y^{1}(t)\ges Y^{2}(t), \q a.s.,\ a.e.\ t\in[0,T].
   \end{equation}
\end{proposition}

\begin{proof}
Fix $ t\in[0,T]$. For  $i=1,2,$ let $(\lambda^{i}(t,\cdot),\mu^{i}(t,\cdot))$ be the solution of the following BDSDE (parameterized by $t$):
\begin{equation}\label{30}
   \lambda^{i}(t,r)=\psi^{i}(t) + \int_r^T f^{i}(t,s,\mu^{i}(t,s)) ds
                    + \int_r^T g(t,s,\mu^{i}(t,s)) d\overleftarrow{B}(s) - \int_r^T \mu^{i}(t,s) dW(s), \qq r\in [t,T].
\end{equation}
By Proposition \ref{4}, we have
\begin{equation}\label{31}
   \lambda^{1}(t,r)\ges  \lambda^{2}(t,r), \q a.s.  \ r\in[t,T].
\end{equation}
Now let
\begin{equation*}
   Y^{i}(t)=\lambda^{i}(t,t), \qq Z^{i}(t,s)=\mu^{i}(t,s), \qq \forall  (t,s) \in \Delta.
\end{equation*}
Then $(Y^{i}(\cdot),Z^{i}(\cdot,\cdot))$ is the solution of (\ref{30}). Finally, by sending $r\downarrow t$ in (\ref{31}), we get the result (\ref{32}).
\end{proof}

Returning to BDSVIE (\ref{28}), we have the following result.

\begin{theorem}\label{43}\sl
For $i=1,2,$ let $f^{i}(t,s,y,z)$ and $g(t,s,y,z)$ satisfy (H2). Suppose $\overline{f}=\overline{f}(t,s,y,z)$ satisfies (H2) and for any $(t,s,z)\in \Delta\times \mathbb{R}$, $\overline{f}(t,s,\cdot,z)$ is increasing, i.e.,
$\overline{f}(t,s,y_{1},z)\ges \overline{f}(t,s,y_{2},z)$, if $y_{1}\ges y_{2}$ with $y_{1},y_{2}\in \mathbb{R}$.
Moreover,
\begin{equation*}
  f^{1}(t,s,y,z)\ges \overline{f}(t,s,y,z)\geq f^{2}(t,s,y,z), \qq
   \forall (t,y,z)\in [0,s]\times \mathbb{R}\times \mathbb{R}, \ \ \text{a.s.,  a.e.} \ s\in[0,T].
\end{equation*}
Then for any $\psi^{i}(\cdot)\in L_{\mathcal{F}_{T}}^2(0,T;\mathbb{R})$ satisfying
   $\psi^{1}(t)\ges \psi^{2}(t), \ a.s. \  t\in[0,T],$
the corresponding unique solution $(Y^{i}(\cdot),Z^{i}(\cdot,\cdot))\in \mathcal{H}_{\Delta}^{2}[0,T]$ of (\ref{28}) satisfy
   \begin{equation*}
   Y^{1}(t)\ges Y^{2}(t), \qq \text{a.s.,  a.e.} \ t\in[0,T].
   \end{equation*}
\end{theorem}

\begin{proof}
Let $\overline{\psi}(\cdot)\in L_{\mathcal{F}_{T}}^2(0,T;\mathbb{R})$ such that
\begin{equation*}
  \psi^{1}(t)\ges \overline{\psi}(t)\ges \psi^{2}(t), \qq \text{a.s.,  a.e.} \   t\in[0,T].
\end{equation*}
It is easy to see that the following equation
\begin{equation*}
  \overline{Y}(t)=\overline{\psi}(t) + \int_t^T \overline{f}(t,s,\overline{Y}(s),\overline{Z}(t,s)) ds
   + \int_t^T g(t,s,\overline{Y}(s),\overline{Z}(t,s)) d\overleftarrow{B}(s)  - \int_t^T \overline{Z}(t,s) dW(s)
\end{equation*}
admits a unique solution $(\overline{Y}(\cdot),\overline{Z}(\cdot,\cdot))\in \mathcal{H}_{\Delta}^{2}[0,T]$.
Set $\widetilde{Y}_{0}(\cdot)=Y^{1}(\cdot)$ and consider the following BDSVIE:
\begin{equation*}
  \widetilde{Y}_{1}(t)=\overline{\psi}(t) + \int_t^T \overline{f}(t,s,\widetilde{Y}_{0}(s),\widetilde{Z}_{1}(t,s)) ds
   + \int_t^T g(t,s,\widetilde{Y}_{0}(s),\widetilde{Z}_{1}(t,s)) d\overleftarrow{B}(s)  - \int_t^T \widetilde{Z}_{1}(t,s) dW(s).
\end{equation*}
Let $(\widetilde{Y}_{1}(\cdot),\widetilde{Z}_{1}(\cdot,\cdot))\in \mathcal{H}_{\Delta}^{2}[0,T]$ be the unique solution to the above equation. Since
$$\left\{\ba{ll}
\ds   f^{1}(t,s,\widetilde{Y}_{0}(s),z)\ges\overline{f}(t,s,\widetilde{Y}_{0}(s),z),
      \q  (t,z)\in [0,s]\times \mathbb{R},  \q \text{a.s., \ a.e.} \ s\in[0,T];\\
\ns\ns\ds  \psi^{1}(t) \ges \overline{\psi}(t),\q \text{a.s., \ a.e.} \  t\in[0,T].
\ea\right.$$
From Proposition \ref{35}, we obtain that
\begin{equation*}
  Y^{1}(t)=\widetilde{Y}_{0}(t) \ges \widetilde{Y}_{1}(t), \ \ \ \text{a.s., \ a.e.} \  t\in[0,T].
\end{equation*}
Next, we consider the following BDSVIE:
\begin{equation*}
  \widetilde{Y}_{2}(t)=\overline{\psi}(t) + \int_t^T \overline{f}(t,s,\widetilde{Y}_{1}(s),\widetilde{Z}_{2}(t,s)) ds
   + \int_t^T g(t,s,\widetilde{Y}_{1}(s),\widetilde{Z}_{2}(t,s)) d\overleftarrow{B}(s)  - \int_t^T \widetilde{Z}_{2}(t,s) dW(s).
\end{equation*}
We see that the above equation admits a unique solution in $\mathcal{H}_{\Delta}^{2}[0,T]$ and denote it by
 $(\widetilde{Y}_{2}(\cdot),\widetilde{Z}_{2}(\cdot,\cdot))$. Since $y\mapsto \overline{f}(t,s,y,z)$ is increasing, we have
\begin{equation*}
   \overline{f}(t,s,\widetilde{Y}_{0}(s),z) \ges \overline{f}(t,s,\widetilde{Y}_{1}(s),z),
    \q (t,z)\in [0,s]\times \mathbb{R},  \q \text{a.s., \  a.e.} \ s\in[0,T].
\end{equation*}
Hence, similar to the above discussion, we obtain
\begin{equation*}
  \widetilde{Y}_{1}(t)\ges \widetilde{Y}_{2}(t), \q \text{a.s., \ a.e.} \ t\in[0,T].
\end{equation*}
By induction, we can construct a sequence
$\{(\widetilde{Y}_{k}(\cdot),\widetilde{Z}_{k}(\cdot,\cdot))\}_{k\ges 1}\in \mathcal{H}_{\Delta}^{2}[0,T]$
such that
\begin{equation*}
\begin{split}
  \widetilde{Y}_{k}(t)=&\overline{\psi}(t) + \int_t^T \overline{f}(t,s,\widetilde{Y}_{k-1}(s),\widetilde{Z}_{k}(t,s)) ds\\
   &+ \int_t^T g(t,s,\widetilde{Y}_{k-1}(s),\widetilde{Z}_{k}(t,s)) d\overleftarrow{B}(s)
   - \int_t^T \widetilde{Z}_{k}(t,s) dW(s),\qq t\in[0,T].
\end{split}
\end{equation*}
Similar to the above discussion, we deduce that
\begin{equation*}
  Y^{1}(t)=\widetilde{Y}_{0}(t)\ges \widetilde{Y}_{1}(t)\ges \widetilde{Y}_{2}(t)\ges \cdots \ges \widetilde{Y}_{k}(t)\ges \cdots, \q\text{ a.s., \ a.e.} \ t\in[0,T].
\end{equation*}
 In the following we show that the sequence $\{(\widetilde{Y}_{k}(\cdot),\widetilde{Z}_{k}(\cdot,\cdot))\}_{k\ges 1}$  is Cauchy in Banach space $\mathcal{H}_{\Delta}^{2}[0,T]$. To get this, we introduce an equivalent norm of the space
$\mathcal{H}_{\Delta}^{2}[0,T]$ as follows,
\begin{equation*}
  \| (Y(\cdot),Z(\cdot,\cdot)) \|_{\beta}^{2}
  \deq \dbE\int_0^{T} \bigg[e^{\beta t}|Y(t)|^{2}+\int_t^{T} e^{\beta s}|Z(t,s)|^{2} ds\bigg]dt.
\end{equation*}
Using the estimate (\ref{20}), we obtain
\begin{align}
    \dbE&\int_0^T \bigg[e^{\beta t}|\widetilde{Y}_{k}(t)-\widetilde{Y}_{k-1}(t)|^{2}
      + \int_t^T e^{\beta s}|\widetilde{Z}_{k}(t,s)-\widetilde{Z}_{k-1}(t,s)|^{2} ds \bigg]dt \nonumber\\
 \les&\ \frac{10}{\beta}\dbE\int_0^T \int_t^T e^{\beta s}
      \left|\overline{f}(t,s,\widetilde{Y}_{k-1}(s),\widetilde{Z}_{k}(t,s))
      -\overline{f}(t,s,\widetilde{Y}_{k-2}(s),\widetilde{Z}_{k-1}(t,s))\right|^{2} dsdt \nonumber\\
     &+  \dbE \int_0^T \int_t^T e^{\beta s}\left|g(t,s,\widetilde{Y}_{k-1}(s),\widetilde{Z}_{k}(t,s))
      - g(t,s,\widetilde{Y}_{k-2}(s),\widetilde{Z}_{k-1}(t,s))\right|^{2} dsdt \nonumber\\
     &+  \dbE \int_0^T e^{\beta t}\int_t^T\left|g(t,s,\widetilde{Y}_{k-1}(s),\widetilde{Z}_{k}(t,s))
      - g(t,s,\widetilde{Y}_{k-2}(s),\widetilde{Z}_{k-1}(t,s))\right|^{2} dsdt \nonumber\\
 \les&\ \delta \dbE\int_0^T \bigg[e^{\beta t}|\widetilde{Y}_{k-1}(t)-\widetilde{Y}_{k-2}(t)|^{2}
      + \int_t^T e^{\beta s}|\widetilde{Z}_{k}(t,s)-\widetilde{Z}_{k-1}(t,s)|^{2} ds \bigg]dt,  \label{36}
   \end{align}
where $\delta=\frac{K}{\beta}+\alpha(T+2)$ and $K=10c(T+1)+\alpha$.
Note that $\alpha(T+2)<1$, then by choosing $\beta=\frac{4K}{1-2\alpha(T+2)}$, we have
\begin{align}
    \dbE& \int_0^{T}\bigg[e^{\beta t}|\widetilde{Y}_{k}(t)-\widetilde{Y}_{k-1}(t)|^{2}
        + \int_t^{T} e^{\beta s}|\widetilde{Z}_{k}(t,s)-\widetilde{Z}_{k-1}(t,s)|^{2} ds\bigg] dt \nonumber\\
\les&\  \frac{\epsilon}{1-\epsilon} \dbE \int_0^T e^{\beta t}|\widetilde{Y}_{k-1}(t)-\widetilde{Y}_{k-2}(t)|^{2} dt \nonumber\\
\les&\ (\frac{\epsilon}{1-\epsilon})^{k-2} \dbE \int_0^{T}\bigg[e^{\beta t}|\widetilde{Y}_{2}(t)-\widetilde{Y}_{1}(t)|^{2}
        + \int_t^{T} e^{\beta s}|\widetilde{Z}_{2}(t,s)-\widetilde{Z}_{1}(t,s)|^{2} ds\bigg] dt,  \label{45}
\end{align}
where $\epsilon=\frac{1+2\alpha(T+2)}{4}<\frac{1}{2}$.
It follows that $\{(\widetilde{Y}_{k}(\cdot),\widetilde{Z}_{k}(\cdot,\cdot))\}_{k\ges 1}$  is Cauchy in Banach space
$\mathcal{H}_{\Delta}^{2}[0,T]$.  We denote their limits by $\widetilde{Y}(\cdot)$ and
$\widetilde{Z}(\cdot,\cdot)$. Then $(\widetilde{Y}(\cdot),\widetilde{Z}(\cdot,\cdot))\in \mathcal{H}_{\Delta}^{2}[0,T]$ and
\begin{equation*}
  \lim_{k\rightarrow \infty}\dbE\bigg[\int_0^{T} e^{\beta t}|\widetilde{Y}_{k}(t)-\widetilde{Y}(t)|^{2} dt
      + \int_0^T\int_t^T e^{\beta s}|\widetilde{Z}_{k}(t,s)-\widetilde{Z}(t,s)|^{2} ds dt\bigg]=0.
\end{equation*}
Furthermore, we have
\begin{equation*}
  \widetilde{Y}(t)=\overline{\psi}(t) + \int_t^T \overline{f}(t,s,\widetilde{Y}(s),\widetilde{Z}(t,s)) ds
   + \int_t^T g(t,s,\widetilde{Y}(s),\widetilde{Z}(t,s)) dB(s) - \int_t^T \widetilde{Z}(t,s) dW(s).
\end{equation*}
Hence, by the existence and uniqueness of BDSVIE, we deduce
\begin{equation*}
   Y^{1}(t)=\widetilde{Y}_{0}(t)\ges \widetilde{Y}(t)=\overline{Y}(t), \qq \text{a.s., \ a.e.} \ t\in[0,T].
\end{equation*}
Similarly, we can prove that
\begin{equation*}
 \overline{Y}(t)\ges Y^{2}(t),  \qq \text{a.s., \ a.e.} \ t\in[0,T].
\end{equation*}
Therefore, our conclusion follows.
\end{proof}

\begin{remark}\rm
One may curious that, if the generators $f(\cd)$ and $g(\cd)$ depend on the term $Z(s,t)$, can the comparison theorem still hold? The answer is positive for some special case. However, due to the technical problem, the general case of $f(\cd)$ and $g(\cd)$ depending on $Z(s,t)$ may not be obtained up to now (see Wang and Yong \cite{Yong}, and Wang, Sun and Yong \cite{Wang-Sun-Yong2018}). So we prefer not to discuss this situation here.
\end{remark}

\begin{example}\rm
 Suppose we are facing with the following two BDSVIEs:
 \begin{equation*}
 \begin{split}
   Y^{1}(t)=&\ \psi^{1}(t) + \int_t^T \(|Y^{1}(s)|+|Z^{1}(t,s)|\) ds
            + \int_t^T \cos \big(Z^{1}(t,s)\big) d\overleftarrow{B}(s)- \int_t^T Z^{1}(t,s) dW(s);\\
   Y^{2}(t)=&\ \psi^{2}(t) - \int_t^T \(|Y^{2}(s)|+|Z^{2}(t,s)|\) ds
            + \int_t^T \cos \big(Z^{1}(t,s)\big) d\overleftarrow{B}(s)- \int_t^T Z^{2}(t,s) dW(s),
\end{split}
\end{equation*}
 where $\psi^{1}(t)\ges \psi^{2}(t),$ a.s., $\forall t\in [0,T]$. Then if we choose $\overline{f}(t,s,y,z)=y+z$, according to Theorem \ref{43}, we get
 $$ Y^{1}(t)\ges  Y^{2}(t), \qq \text{a.s.,\ a.e.}\ t\in[0,T].$$
\end{example}

\subsection{BDSVIE with continuous coefficients}

As an application of the comparison theorem, this subsection is dedicated to the study of one-dimensional BDSVIEs with continuous coefficients of the following type:
\begin{equation}\label{39}
   Y(t)=\psi(t) + \int_t^T f(t,s,Y(s),Z(t,s)) ds
         + \int_t^T g(t,s,Y(s),Z(t,s)) d\overleftarrow{B}(s) - \int_t^T Z(t,s) dW(s),\qq t\in[0,T],
\end{equation}
where
$f,g:\Omega\times \Delta\times \mathbb{R}\times \mathbb{R}\rightarrow \mathbb{R}$
are jointly measurable such that for all $(y,z)\in \mathbb{R}\times \mathbb{R}$,
$f(\cdot,\cdot,y,z)$ and $g(\cdot,\cdot,y,z)$  belong to $L_{\dbF}^{2}(\Delta; \mathbb{R})$.
\begin{itemize}
  \item [$\mathbf{(H4)}$]
        There exist some constants $M>0$ and $0<\alpha <\frac{1}{T+2}$ such that for all $(\omega,t,s,y,z)\in \Omega\times \Delta \times\mathbb{R}\times \mathbb{R}$,
        the following items hold:
        \begin{itemize}
          \item [(1)]  Linear growth:  $ |f(t,s,y,z)|\les M(1+|y|+|z|);$
          \item [(2)]  $f(t,s,y,z)$ is continuous with respect to $(t,y,z)$ and $g(t,s,y,z)$ is continuous with respect to $t$;
          \item [(3)]  $ \big|g(t,s,y,z)-g(t,s,y',z')\big|^{2}\les \alpha(|y-y'|^{2} + |z-z'|^{2}).$
        \end{itemize}
\end{itemize}

\begin{theorem}\label{44}\sl
Under the assumption (H4), if $\psi(\cdot)\in S_{\mathcal{F}_{T}}^2(0,T;\mathbb{R})$, then BDSVIE (\ref{39}) has a solution
 $(Y(\cd),Z(\cd,\cd)) \in \mathcal{H}_{\Delta}^{2}[0,T]$. Furthermore, there is a minimal solution  $(Y^{*}(\cd),Z^{*}(\cd,\cd))$ of (\ref{39}) in the sense that, for any other solution $(Y(\cd),Z(\cd,\cd))$ of (\ref{39}), we have
$Y^{*}(t)\les Y(t), \ \text{a.s., a.e.} \ t\in [0,T].$
\end{theorem}
For fixed $(\omega,t,s)\in \Omega\times \Delta$, define the sequence $f_{n}(\omega,t,s,y,z)$ associated to $f$ as in Proposition \ref{40} by:
\begin{equation*}
  f_{n}(\omega,t,s,y,z)\deq\inf_{y',z'\in \mathbb{Q}} \bigg\{f(\omega,t,s,y',z')+n(|y-y'|+|z-z'|)\bigg\},
\end{equation*}
then, for $n\ges M$, $f_{n}$ is jointly measurable and uniformly linear growth in $y,z$ with constant $M$.
We also define the following function,
\begin{equation*}
  F(\omega,t,s,y,z) = M(1+|y| + |z|).
\end{equation*}
Given $\psi(\cdot)\in S_{\mathcal{F}_{T}}^2(0,T;\mathbb{R})$, from Theorem \ref{22}, there exist two pairs of processes
$(Y^{n},Z^{n})\in \mathcal{H}_{\Delta}^{2}[0,T]$ and $(U,V)\in \mathcal{H}_{\Delta}^{2}[0,T]$, which are the solutions to the following
Equations, respectively,
\begin{equation}\label{41}
   Y^{n}(t)=\psi(t) + \int_t^T f_{n}(t,s,Y^{n}(s),Z^{n}(t,s)) ds
    + \int_t^T g(t,s,Y^{n}(s),Z^{n}(t,s)) d\overleftarrow{B}(s) - \int_t^T Z^{n}(t,s) dW(s),
\end{equation}
\begin{equation*}
 \q\ \  U(t)=\psi(t) + \int_t^T F(t,s,U(s),V(t,s)) ds
   + \int_t^T g(t,s,U(s),V(t,s)) d\overleftarrow{B}(s) - \int_t^T V(t,s) dW(s). \ \ \ \ \
\end{equation*}
From Proposition \ref{40} and Theorem \ref{43}, we get
\begin{equation*}
  \forall n\ges m\ges M, \qq Y^{m}(t)\les Y^{n}(t)\les U(t), \qq \text{a.s.,\ a.e.} \ t\in [0,T].
\end{equation*}
Note that from the definition, $f_{n}(\cd)$ satisfies the assumption (H2), then
it's easy to obtain the follows result from the estimate (\ref{18}).
\begin{lemma}\rm
There exists a constant $C>0$ depending only on $M,\ c,\ \alpha,\ T$ and $\psi$, such that
\begin{equation*}
  \forall n\ges M, \qq \|(Y^{n},Z^{n})\|_{\mathcal{H}_{\Delta}^{2}[0,T]}\les C, \qq \|(U,V)\|_{\mathcal{H}_{\Delta}^{2}[0,T]}\les C.
\end{equation*}
\end{lemma}
\begin{lemma}\rm
$\{(Y^{n}(\cd),Z^{n}(\cd,\cd))\}_{n\ges1}$ converges in $\mathcal{H}_{\Delta}^{2}[0,T]$.
\end{lemma}
\begin{proof}[Sketch of proof]
Let $n_{0} \ges M$.
Since $\{Y^{n}(\cd)\}$ is increasing and bounded in $L_{\dbF}^{2}(0,T;\mathbb{R})$, we infer from the dominated convergence theorem that $Y^{n}(\cd)$ converges in $L_{\dbF}^{2}(0,T;\mathbb{R})$. We denote by $Y(\cd)$ the limit of $\{Y^{n}(\cd)\}$. For $n\ges n_{0}$, by the estimate (\ref{20}), similarly to (\ref{36}) and (\ref{45}), note that $f_{n}(\cd)$ are uniformly linear growth and $\{Y^{n}(\cd),Z^{n}(\cd,\cd)\}$ are bounded, we deduce
\begin{equation*}
\begin{split}
    \dbE&  \int_0^{T} \bigg(e^{\beta t}|Y^{n+1}(t)-Y^{n}(t)|^{2}
        + \int_t^{T} e^{\beta s}|Z^{n+1}(t,s)-Z^{n}(t,s)|^{2} ds\bigg) dt\\
\les& \ \rho \dbE  \int_0^{T} \bigg(e^{\beta t}|Y^{n_{0}+1}(t)-Y^{n_{0}}(t)|^{2}
        + \int_t^{T} e^{\beta s}|Z^{n_{0}+1}(t,s)-Z^{n_{0}}(t,s)|^{2} ds\bigg) dt,\\
\end{split}
\end{equation*}
where $\rho=(\frac{\epsilon}{1-\epsilon})^{n-n_{0}}<1, $ and $\epsilon=\frac{1+2\alpha(T+2)}{4}.$
Let $n\rightarrow \infty$, from which the result follows.
\end{proof}

Now we give the proof of Theorem \ref{44}.

\begin{proof}[Proof of Theorem \ref{44}]
On the one hand, for all $n\ges n_{0}\ges M$, one has $Y^{n_{0}}(t)\leq Y^{n}(t)\leq U(t)$,
 and $\{Y^{n}(t)\}$ converges to $Y(t)$ in $L_{\dbF}^{2}(0,T;\mathbb{R})$, $a.s. \ t\in [0,T].$

\ms

 On the other hand, since $Z^{n}(\cd,\cd)$ converges to $Z(\cd,\cd)$ in $L_{\dbF}^2(\Delta;\mathbb{R}),$ we can assume, choosing a subsequence if needed, that $Z^{n}(t,s)\rightarrow Z(t,s)$ and $\overline{G}\deq\sup_{n}|Z^{n}(t,s)|$ is integrable, a.s., $(t,s)\in\Delta$. Therefore, from (i) and (iv) of Proposition \ref{40}, we get,
\begin{equation*}
\begin{split}
 f_{n}(t,s,Y^{n}(s),Z^{n}(t,s))\xrightarrow { \ n \ }&\ f(t,s,Y(s),Z(t,s)), \ \ \text{a.s.} \ (t,s)\in\Delta,\\
 |f_{n}(t,s,Y^{n}(s),Z^{n}(t,s))| \les&\ M\big(1 + \sup_{n}|Y^{n}(s)| + \sup_{n}|Z^{n}(t,s)|\big)\\
                                    = &\ M\big(1 + \sup_{n}|Y^{n}(s)| + \overline{G}\big).
\end{split}
\end{equation*}
Thus, for $r\in [t,T]$, it holds that
\begin{equation*}
  \int_r^T f_{n}(t,s,Y^{n}(s),Z^{n}(t,s)) ds \xrightarrow { \ n \ } \int_r^T f(t,s,Y(s),Z(t,s)) ds, \ \ \text{a.s.}
\end{equation*}
From the continuity properties of the stochastic integral, we have the following convergence (in probability),
\begin{equation*}
\begin{split}
  &\sup_{t\les r\les T} \bigg| \int_r^T Z^{n}(t,s) dW(s) - \int_r^T Z(t,s) dW(s) \bigg|  \xrightarrow { \ n \ } 0, \\
  &\sup_{t\les r\les T} \bigg| \int_r^T g(t,s,Y^{n}(s),Z^{n}(t,s)) d\overleftarrow{B}(s) - \int_r^T g(t,s,Y(s),Z(t,s)) d\overleftarrow{B}(s)\bigg|  \xrightarrow { \ n \ } 0.
\end{split}
\end{equation*}
Choosing, again, a subsequence, we can assume that the above convergence is $\dbP$-a.s. Finally,
\begin{equation*}
\begin{split}
  |Y^{n}(t)-Y^{m}(t)|
 \les&  \int_t^T \big| f_{n}(t,s,Y^{n}(s),Z^{n}(t,s)) - f_{m}(t,s,Y^{m}(s),Z^{m}(t,s))\big| ds\\
   &+ \bigg| \int_t^T g(t,s,Y^{n}(s),Z^{n}(t,s)) d\overleftarrow{B}(s) - \int_t^T g(t,s,Y^{m}(s),Z^{m}(t,s)) d\overleftarrow{B}(s) \bigg|\\
   &+ \bigg| \int_t^T Z^{n}(t,s) dW(s) - \int_t^T Z^{m}(t,s) dW(s) \bigg|,
\end{split}
\end{equation*}
and taking limits on $m$ and supremum over $t$, we obtain
\begin{equation*}
\begin{split}
     \sup_{0\les t\les T}|Y^{n}(t)-Y(t)|
\les&  \sup_{0\les t\les T} \int_t^T \big| f_{n}(t,s,Y^{n}(s),Z^{n}(t,s)) - f(t,s,Y(s),Z(t,s))\big| ds\\
    &+ \sup_{0\les t\les T}\bigg| \int_t^T g(t,s,Y^{n}(s),Z^{n}(t,s)) d\overleftarrow{B}(s)
     - \int_t^T g(t,s,Y(s),Z(t,s)) d\overleftarrow{B}(s) \bigg|\\
    &+ \sup_{0\les t\les T}\bigg| \int_t^T Z^{n}(t,s) dW(s) - \int_t^T Z(t,s) dW(s) \bigg|, \ \ \dbP\text{-a.s.}
\end{split}
\end{equation*}
From which it follows that $Y^{n}(\cd)$ converges uniformly in $t$ to $Y(\cd)$.
In particular, from (H4) and $\psi(\cdot)\in S_{\mathcal{F}_{T}}^2(0,T;\mathbb{R})$, it's easy to see $Y(\cd)$ is a continuous process. Note that $\{Y^{n}(\cd)\}_{n\ges 1}$ is monotone, therefore, we actually have the uniform convergence for the entire sequence and not just for a subsequence. Taking limits in Eq. (\ref{41}), we deduce that $(Y(\cd),Z(\cd,\cd))\in \mathcal{H}_{\Delta}^{2}[0,T]$ is a solution of (\ref{39}).

\ms

Let $(\widehat{Y}(\cd),\widehat{Z}(\cd,\cd))\in \mathcal{H}_{\Delta}^{2}[0,T]$ be any solution of BDSVIE (\ref{39}).
Then from Theorem \ref{43}, $Y^{n}(t)\les \widehat{Y}(t),$ a.s. $t\in[0,T], \ \forall n\in N$,
and therefore  $Y(t)\les \widehat{Y}(t), \ a.s. \ t\in[0,T],$ which proving that $Y(\cd)$ is the minimal solution. The proof is completed.
\end{proof}

\section{An optimal control problem}

In this section, as an application of BDSVIEs, we study an optimal control problem for doubly stochastic Volterra integral equation. To simplify the presentation, we only discuss the one dimensional case, i.e., $k=l=1$, and we also let $d=1$. Consider the following state equation:
\begin{equation}\label{76}
\begin{split}
  P(t)=&\ \varphi(t) + \int_0^t b(t,s,P(s),Q(s,t),u(s)) ds\\
       &+ \int_0^t \sigma(t,s,P(s),Q(s,t),u(s))dW(s) + \int_0^t Q(t,s) d\overleftarrow{B}(s), \qq t\in [0,T],
\end{split}
\end{equation}
where $u(\cdot)$ is a control process;
$b,\sigma:\Delta^{c}\times \mathbb{R}\times \mathbb{R}\times U\rightarrow \mathbb{R}$ are some given maps and
$\varphi(\cdot)\in L_{\dbF}^2(0,T;\mathbb{R})$; $U$ is a bounded interval in $\mathbb{R}$.
The cost function is defined to be the following Lagrange form:
\begin{equation}\label{77}
  J(u(\cdot))=\dbE \int_0^T h (t,P(t),u(t)) dt,
\end{equation}
where $h:[0,T]\times \mathbb{R}\times U$ is a given map as well.
In the above, all the functions can be random.
We now introduce the following assumption.
It's should be pointed out that the conditions assumed are more than sufficient,
and one can relax many of them. But we prefer these strong conditions to make the presentation simple.
\begin{itemize}
  \item [$\mathbf{(H5)}$]
 Let $b$ and $\sigma$ be continuous in all of their arguments, and differentiable in the variables $p, q$ and $u$, with bounded derivatives.
Also, there exists a constant $C$ such that
\begin{equation*}
  |b(t,s,0,0,u)|+|\sigma(t,s,0,0,u)|\les C, \qq \forall (t,s)\in \Delta^{c}, \ u\in U.
\end{equation*}
\end{itemize}
%
Let
\begin{equation}\label{78}
  \mathcal{U} \deq \{ u:[0,T]\times \Omega \rightarrow U \ | \  u(\cd) \text{ is } \dbF\text{-measurable}, \
  \dbE\int_0^T|u(s)|^2ds<\i \}.
\end{equation}
From Theorem \ref{63}, it is not hard to see that under the assumption (H5), for any $\varphi(\cdot)\in L_{\dbF}^2(0,T;\mathbb{R})$ and $u(\cdot)\in \mathcal{U}$,
Eq. (\ref{76}) has a unique M-solution $(P(\cdot),Q(\cdot,\cdot))\in \mathcal{H}^{2}[0,T]$.
Thus the cost functional $J(u(\cdot))$ is well-defined. Our optimal control problem can be stated as follows:

\ms

$\mathbf{Problem \ (C)}.$ Find a $\overline{u}(\cdot)\in \mathcal{U}$ such that
\begin{equation}\label{79}
  J(\overline{u}(\cdot))=\inf_{\overline{u}(\cdot)\in \mathcal{U}} J(u(\cdot)).
\end{equation}
Any $\overline{u}(\cdot)$ satisfying (\ref{79}) is called an $optimal \ control$ of Problem (C),
the corresponding state process $(\overline{P}(\cdot),\overline{Q}(\cdot,\cdot))$ is called an optimal state process and
$(\overline{u}(\cdot),\overline{P}(\cdot),\overline{Q}(\cdot,\cdot))$ is called an $optimal \ pair$.

\ms

Next result is called the $duality \ principle$ of linear doubly stochastic Volterra integral equation, which plays an important role below.
\begin{theorem}\label{80}\sl
For $i=1,...,4,$, suppose $A_{i}(\cdot,\cdot)\in L^{2}_{\dbF}([0,T]^{2};\mathbb{R})$,
$\varphi(\cdot)\in L_{\dbF}^2(0,T;\mathbb{R})$ and $\psi(\cdot)\in L_{\mathcal{F}_{T}}^2(0,T;\mathbb{R})$.
Let $(P(\cdot),Q(\cdot,\cdot))\in \mathcal{H}^{2}[0,T]$ be the M-solution of the following FDSVIE:
\begin{equation}\label{81}
\begin{split}
  P(t)=& -\varphi(t) + \int_0^t [A_{1}(t,s)P(s) + A_{2}(t,s)Q(s,t)] ds\\
       &  + \int_0^t[A_{3}(t,s)P(s) + A_{4}(t,s)Q(s,t)] dW(s) + \int_0^t q(t,s) d\overleftarrow{B}(s), \qq t\in [0,T],
\end{split}
\end{equation}
and
$(Y(\cdot),Z(\cdot,\cdot))\in \mathcal{H}^{2}[0,T]$ be the M-solution to the following BDSVIE:
\begin{equation}\label{82}
\begin{split}
  Y(t)=& -\psi(t) + \int_t^T [A_{1}(s,t)Y(s) + A_{3}(s,t)Z(s,t)] ds\\
       &+\int_t^T [A_{2}(s,t)Y(s)+A_{4}(s,t)Z(s,t)]d\overleftarrow{B}(s)-\int_t^T Z(t,s)dW(s), \qq t\in [0,T].
\end{split}
\end{equation}
Then the following relation holds:
\begin{equation}\label{83}
  \dbE\int_0^T \psi(t)P(t)dt = \dbE\int_0^T \varphi(t)Y(t) dt.
\end{equation}
\end{theorem}

\begin{proof}
Observe the following,
$$\2n\ba{ll}
\ds \dbE\int_0^T \psi(t)P(t)dt\\
\ns\ds = \dbE\int_0^T \bigg( -Y(t) + \int_t^T [A_{1}(s,t)Y(s) + A_{3}(s,t)Z(s,t)] ds\\
\ns\ds\qq\qq\ \ +\int_t^T[A_{2}(s,t)Y(s)+A_{4}(s,t)Z(s,t)]d\overleftarrow{B}(s)-\int_t^T Z(t,s)dW(s)\bigg)P(t)dt\\
\ns\ds= \dbE\int_0^T Y(t)\bigg(-P(t)+\int_0^t A_{1}(t,s)P(s)ds\bigg)dt
    +\dbE\int_0^T\int_0^t A_{3}(t,s)Z(t,s)P(s)dsdt\\
\ns\ds\q +\dbE\int_0^T \bigg(\int_t^T [A_{2}(s,t)Y(s) + A_{4}(s,t)Z(s,t)] d\overleftarrow{B}(s)\bigg)P(t)dt  \\
\ns\ds= \dbE\int_0^T Y(t)\bigg(-P(t)+\int_0^t A_{1}(t,s)P(s)ds\bigg)dt
    +\dbE\int_0^T \bigg(\int_0^t Z(t,s) dW(s)\bigg) \cdot \bigg(\int_0^t A_{3}(t,s)P(s)dW(s)\bigg)dt\\
\ns\ds\q +\dbE\int_0^T \bigg(\int_t^T [A_{2}(s,t)Y(s) + A_{4}(s,t)Z(s,t)]
   d\overleftarrow{B}(s)\bigg)\cdot\bigg(\dbE[P(t)]+\int_t^T Q(t,s) d\overleftarrow{B}(s)\bigg) dt \\
\ns\ds = \dbE\int_0^T Y(t)\bigg(-P(t)+\int_0^t A_{1}(t,s)P(s)ds\bigg)dt
    +\dbE\int_0^T \big(Y(t)- \dbE[Y(t)]\big) \cdot \bigg(\int_0^t A_{3}(t,s)P(s)dW(s)\bigg)dt\\
\ns\ds\q +\dbE\int_0^T Y(t)\bigg(\int_0^t A_{2}(t,s)Q(s,t)ds\bigg)dt+\dbE\int_0^T \int_0^t A_{4}(t,s)Z(t,s)Q(s,t)dsdt \\
\ns\ds=\dbE\int_0^T Y(t)\bigg(-P(t)+\int_0^t[A_{1}(t,s)P(s)+A_{2}(t,s)Q(s,t)]ds+\int_0^tA_{3}(t,s)P(s)dW(s)\bigg)dt \\
\ns\ds\q + \dbE\int_0^T \bigg(\int_0^t Z(t,s) dW_{s}\bigg) \cdot \bigg(\int_0^t A_{4}(t,s)Q(s,t) dW(s)\bigg)dt \\
\ns\ds= \dbE\int_0^T Y(t)\bigg(-P(t)+\int_0^t [A_{1}(t,s)P(s)+A_{2}(t,s)Q(s,t)]ds\\
\ns\ds\qq\qq\qq\ \ + \int_0^t [A_{3}(t,s)P(s)+ A_{4}(t,s)Q(s,t)] dW(s)\bigg)dt
  =\dbE\int_0^T \varphi(t)Y(t) dt.
\ea$$
Thus the relation (\ref{83}) holds.
\end{proof}

Based on the above duality principle, we could establish the following theorem called Pontryagin's maximum principle.
For simplicity, denote
\begin{equation*}
  \overline{b}_{i}(s,t)=b_{i}(s,t,\overline{P}(t),\overline{Q}(t,s),\overline{u}(t)), \qq
  \overline{\sigma}_{i}(s,t)=\sigma_{i}(s,t,\overline{P}(t),\overline{Q}(t,s),\overline{u}(t)),\qq t,s\in[0,T],
\end{equation*}
where $\overline{b}_{i}(\cd)$ (or $\overline{\sigma}_{i}(\cd)$) denotes the partial derivative of $\overline{b}(\cd)$ (or $\overline{\sigma}_{i}(\cd)$) at $i$ with $i=\bar{p},\bar{q},\bar{u}$, respectively.
\begin{theorem}\label{TH5} \sl
Let the assumption (H5) hold,
and let $(\overline{P}(\cdot),\overline{Q}(\cdot,\cdot),\overline{u}(\cdot))$ be an optimal pair of Problem (C).
Then there exists a unique M-solution $(Y(\cdot),Y_{0}(\cdot);$ $Z(\cdot,\cdot),Z_{0}(\cdot,\cdot))$ of the following BDSVIE:
\bel{84}\left\{\ba{ll}
\ds Y(t)=h_{p}(t,\overline{P}(t),\overline{u}(t)) + \int_t^T \big[\overline{b}_{p}(s,t)Y(s) +\overline{\sigma}_{p}(s,t)Z(s,t)\big] ds\\
\ns\ns\ds\qq\q\ +\int_t^T \big[\overline{b}_{q}(s,t)Y(s)
+\overline{\sigma}_{q}(s,t)Z(s,t)\big] d\overleftarrow{B}(s)- \int_t^T Z(t,s) dW(s),\\
\ns\ns\ds Y_{0}(t)= \int_t^T \big[\overline{b}_{u}(s,t)Y(s) +\overline{\sigma}_{u}(s,t)Z(s,t)\big] ds - \int_t^T Z_{0}(t,s) dW(s),
\ea\right.\ee
such that
\begin{equation}\label{85}
   \big[Y_{0}(t) + h_{u}(t,\overline{p}(t),\overline{u}(t))\big] \big(u - \overline{u}(t) \big) \geq 0, \qq \forall u\in U, \ t\in[0,T], \ a.s.
\end{equation}
\end{theorem}
Note that the second equation in (\ref{84}) is in fact a BSVIE in which the free term is zero, and the drift term does not depend on
$(Y_{0}(\cdot),Z_{0}(\cdot,\cdot))$.

\begin{proof}[Proof of Theorem \ref{TH5}]
Take any $u(\cdot)\in \mathcal{U}$.
Since $\mathcal{U}$ is convex, then for any $\varepsilon \in (0,1)$,
\begin{equation*}
  u_{\varepsilon}(\cdot)\deq \overline{u}(\cdot) + \varepsilon[u(\cdot) - \overline{u}(\cdot)] \in \mathcal{U}.
\end{equation*}
Let $(P_{\varepsilon}(\cdot),Q_{\varepsilon}(\cdot,\cdot))$ be the M-solution of (\ref{76}) corresponding to $ u_{\varepsilon}(\cdot)$. Now for $t,s\in[0,T]$, define
\begin{equation*}
  \xi_{\varepsilon}(t)\deq \frac{P_{\varepsilon}(t)-\overline{P}(t)}{\varepsilon}, \ \ \ \
 \eta_{\varepsilon}(t,s)\deq \frac{Q_{\varepsilon}(t,s)-\overline{Q}(t,s)}{\varepsilon}. \ \
\end{equation*}
Then $(\xi_{\varepsilon}(\cdot),\eta_{\varepsilon}(\cdot,\cdot))\rightarrow (\xi(\cdot),\eta(\cdot,\cdot))$
in $\mathcal{H}^{2}[0,T]$ with $(\xi(\cdot),\eta(\cdot,\cdot))$ satisfying the following equation:
\begin{equation*}\label{}
\begin{split}
  \xi(t)=& \int_0^t \bigg[b_{p}(t,s,\overline{P}(s),\overline{Q}(s,t),\overline{u}(s))\xi(s)
         + b_{q}(t,s,\overline{P}(s),\overline{Q}(s,t),\overline{u}(s))\eta(s,t)\\
      & \ \ \ \ \ \ \ \ \ + b_{u}(t,s,\overline{P}(s),\overline{Q}(s,t),\overline{u}(s))\big(u(s) - \overline{u}(s)\big) \bigg] ds\\
        &+ \int_0^t \bigg[\sigma_{p}(t,s,\overline{P}(s),\overline{Q}(s,t),\overline{u}(s))\xi(s)
         + \sigma_{q}(t,s,\overline{P}(s),\overline{Q}(s,t),\overline{u}(s))\eta(s,t) \\
        & \ \ \ \ \ \ \ \ \ + \sigma_{u}(t,s,\overline{P}(s),\overline{Q}(s,t),\overline{u}(s))\big(u(s) - \overline{u}(s)\big) \bigg] dW(s)
         + \int_0^t \eta(t,s) d\overleftarrow{B}(s)\\
  \equiv& \ \overline{\varphi}(t)  + \int_0^t \eta(t,s) d\overleftarrow{B}(s)\\
 & + \int_0^t \bigg[b_{p}(t,s,\overline{P}(s),\overline{Q}(s,t),\overline{u}(s))\xi(s)
         + b_{q}(t,s,\overline{P}(s),\overline{Q}(s,t),\overline{u}(s))\eta(s,t) \bigg] ds\\
 & + \int_0^t \bigg[\sigma_{p}(t,s,\overline{P}(s),\overline{Q}(s,t),\overline{u}(s))\xi(s)
         + \sigma_{q}(t,s,\overline{P}(s),\overline{Q}(s,t),\overline{u}(s))\eta(s,t)  \bigg] dW(s),
\end{split}
\end{equation*}
where
\begin{equation*}
\begin{split}
  \overline{\varphi}(t)=& \int_0^t b_{u}(t,s,\overline{P}(s),\overline{Q}(s,t),\overline{u}(s))\big(u(s) - \overline{u}(s)\big)ds\\
                  & + \int_0^t \sigma_{u}(t,s,\overline{P}(s),\overline{Q}(s,t),\overline{u}(s))\big(u(s) - \overline{u}(s)\big)dW_{s}, \qq t\in[0,T].
\end{split}
\end{equation*}
Now, let  $(Y(\cdot),Y_{0}(\cdot);Z(\cdot,\cdot),Z_{0}(\cdot,\cdot))$ be the unique M-solution to BDSVIE (\ref{84}).
By the optimality of $(\overline{P}(\cdot),\overline{Q}(\cdot,\cdot),\overline{u}(\cdot))$, and the duality principle (Theorem \ref{80}), we have
\begin{equation*}
\begin{split}
0\les& \frac{J(u_{\varepsilon}(\cdot))-J(\overline{u}(\cdot))}{\varepsilon}\\
 \rightarrow& \dbE \int_0^T \left[h_{p}(t,\overline{P}(t),\overline{u}(t))\xi(t)
                               + h_{u}(t,\overline{P}(t),\overline{u}(t))\big(u(s) - \overline{u}(s)\big)\right] dt\\
   =&  \dbE\int_0^T \left(\overline{\varphi}(t)Y(t) + h_{u}(t,\overline{P}(t),\overline{u}(t))\big(u(s) - \overline{u}(s)\big)\right) dt\\
   =&  \dbE\int_0^T \left( \int_0^t b_{u}(t,s,\overline{P}(s),\overline{Q}(s,t),\overline{u}(s))\big(u(s) - \overline{u}(s)\big)ds \right)Y(t)dt\\
    & +\dbE\int_0^T \left( \int_0^t \sigma_{u}(t,s,\overline{P}(s),\overline{Q}(s,t),\overline{u}(s))\big(u(s) - \overline{u}(s)\big)dW(s) \right)Y(t)dt\\
    & +\dbE\int_0^T h_{u}(t,\overline{P}(t),\overline{u}(t))\big(u(s) - \overline{u}(s)\big) dt \\
   =&  \dbE\int_0^T \left( \int_t^T b_{u}(s,t,\overline{P}(t),\overline{Q}(t,s),\overline{u}(t))Y(s)ds\right)
       \big(u(t) - \overline{u}(t)\big)dt\\
    & +\dbE\int_0^T \int_0^t \sigma_{u}(t,s,\overline{P}(s),\overline{Q}(s,t),\overline{u}(s))Z(t,s)\big(u(s) - \overline{u}(s)\big)dsdt\\
    & +\dbE\int_0^T h_{u}(t,\overline{P}(t),\overline{u}(t))\big(u(s) - \overline{u}(s)\big) dt \\
   =&  \dbE\int_0^T \bigg(\int_t^T \big[b_{u}(s,t,\overline{P}(t),\overline{Q}(t,s),\overline{u}(t))Y(s)\\
    &  \ \ \ \ \ \ \ \ \ \ \ \ \ \ \ \ \   +\sigma_{u}(s,t,\overline{P}(t),\overline{Q}(t,s),\overline{u}(t))Z(s,t)\big] ds
           + h_{u}(t,\overline{P}(t),\overline{u}(t))\bigg)\big(u(t) - \overline{u}(t)\big) dt\\
   =&  \dbE\int_0^T \big[Y_{0}(t) + h_{u}(t,\overline{P}(t),\overline{u}(t))\big] \big(u(t) - \overline{u}(t) \big) dt.
\end{split}
\end{equation*}
Since the above holds for all $u(\cdot)\in \mathcal{U}$, we obtain (\ref{85}).
\end{proof}
Note that if we define
\begin{equation}\label{86}
\begin{split}
 H(t,&\overline{P}(t),\overline{Q}(\cdot,t),\overline{u}(t),Y(t),Z(\cdot,t),u)
  \deq -\big[Y_{0}(t) + h_{u}(t,\overline{P}(t),\overline{u}(t))\big]u\\
  =& -\dbE \bigg\{h_{u}(t,\overline{P}(t),\overline{u}(t)) + \int_t^T \bigg[b_{u}(s,t,\overline{P}(t),\overline{Q}(t,s),\overline{u}(t))Y(s)\\
   & \ \ \ \ \ \ \ \  +\sigma_{u}(s,t,\overline{P}(t),\overline{Q}(t,s),\overline{u}(t))Z(s,t)\bigg] ds \bigg| \mathcal{F}_{t} \bigg\}u,
\end{split}
\end{equation}
then (\ref{85}) can be written as
\begin{equation}\label{87}
   H(t,\overline{P}(t),\overline{Q}(\cdot,t),\overline{u}(t),Y(t),Z(\cdot,t),\overline{u}(t))
=  \max_{u\in U} H(t,\overline{P}(t),\overline{Q}(\cdot,t),\overline{u}(t),Y(t),Z(\cdot,t),u).
\end{equation}
We call $H(\cdot)$ defined by (\ref{86}) the Hamiltonian of our optimal control problem, call (\ref{85}) (and (\ref{87})) the maximum condition, and call the first BDSVIE in (\ref{84}) the adjoint equation of FDSVIE (\ref{76}), along the optimal pair $(\overline{P}(\cdot),\overline{Q}(\cdot,\cdot),\overline{u}(\cdot))$.

\ms

Finally, by putting (\ref{76}), (\ref{84}) and (\ref{85}) together (dropping the bars in
$(\overline{P}(\cdot),\overline{Q}(\cdot,\cdot),\overline{u}(\cdot))$), we obtain the following system:
\bel{88}\left\{\ba{ll}
\ds P(t)=\varphi(t) + \int_0^t b(t,s,P(s),Q(s,t),u(s)) ds \\
\ns\ns\ds\qq\q + \int_0^t \sigma(t,s,P(s),Q(s,t),u(s))dW(s) + \int_0^t Q(t,s) d\overleftarrow{B}(s),\\
\ns\ns\ds Y(t)=h_{p}(t,P(t),u(t))-\int_t^T Z(t,s) dW(s)\\
\ns\ns\ds\qq\q  + \int_t^T \bigg(b_{p}(s,t,P(t),Q(t,s),u(t))Y(s)
+ \sigma_{p}(s,t,P(t),Q(t,s),u(t))Z(s,t)\bigg) ds\\
\ns\ns\ds\qq\q +\int_t^T \bigg(b_{q}(s,t,P(t),Q(t,s),u(t))Y(s)
+\sigma_{q}(s,t,P(t),Q(t,s),u(t))Z(s,t)\bigg) d\overleftarrow{B}(s),\\
\ns\ns\ds Y_{0}(t)= \int_t^T \bigg(b_{u}(s,t,P(t),Q(t,s),u(t))Y(s) \\
\ns\ns\ds\qq\qq\qq + \sigma_{u}(s,t,P(t),Q(t,s),u(t))Z(s,t)\bigg) ds - \int_t^T Z_{0}(t,s) dW(s),\\

\ns\ns\ds  \big[Y_{0}(t)+h_{u}(t,P(t),u(t))\big]\big(v-u(t)\big)\geq 0, \qq \forall v\in U, \ t\in[0,T], \ \text{a.s.}
\ea\right.\ee
This is a couple of FDSVIE and BDSVIE systems.
The coupling is through the maximum condition(via $u(\cdot)$).
We call (\ref{88}) a forward-backward doubly stochastic Volterra integral equation (FBDSVIE, for short).
Such kinds of equations are still under careful investigation.

\ms

The purpose of presenting a simple optimal control problem of FDSVIEs here is to realize a major motivation of studying BDSVIEs. It is possible to discuss Bolza type cost functional. However, we have no intention to have a full exploration of general optimal control problems for FDSVIEs in the current paper since such kind of general problems are much more involved and they deserve to be addressed in another paper. We will report further results along that line in a forthcoming paper (see Shi, Wen and Xiong \cite{Shi-Wen-Xiong2019}).

\bibliographystyle{elsarticle-num}

\end{document}